\documentclass[11pt]{amsart}
\usepackage{amsmath,amssymb,bm,mathrsfs}
\usepackage[all]{xy}
\usepackage{hyperref}

\newcommand{\map}[1]{\xrightarrow{#1}}
\newcommand{\iso}{\cong}

\newcommand{\Gal}{\mathrm{Gal}}
\newcommand{\Hom}{\mathrm{Hom}}
\newcommand{\Aut}{\mathrm{Aut}}
\newcommand{\End}{\mathrm{End}}
\newcommand{\Spec}{\mathrm{Spec}}
\newcommand{\define}{\stackrel{\mathrm{def}}{=}}

\newcommand{\Q}{\mathbb Q}
\newcommand{\Z}{\mathbb Z}
\newcommand{\R}{\mathbb R}
\newcommand{\C}{\mathbb C}
\newcommand{\F}{\mathbb F}
\newcommand{\A}{\mathbb A}
\newcommand{\co}{\mathcal O}

\newcommand{\alg}{\mathrm{alg}}
\newcommand{\ord}{\mathrm{ord}}

\newcommand{\Lie}{\mathrm{Lie}}

\newcommand{\Falt}{\mathrm{Falt}}

\newcommand{\Colmez}{\mathrm{Col}}

\newcommand{\GSpin}{\mathrm{GSpin}}
\newcommand{\SO}{\mathrm{SO}}
\newcommand{\GSp}{\mathrm{GSp}}
\newcommand{\GL}{\mathrm{GL}}
\newcommand{\SL}{\mathrm{SL}}

\newcommand{\action}{\bullet}
\newcommand{\kk}{{\bm{k}}}
\newcommand{\dR}{{dR}}
\newcommand{\crys}{{crys}}
\newcommand{\Pic}{\mathrm{Pic}}

\newcommand{\weil}{\omega}
\newcommand{\taut}{\bm{\omega}}
\newcommand{\beef}{\diamond}

\begin{document}
\author{Benjamin Howard}
\thanks{}
\title{On the averaged Colmez conjecture}
\thanks{This research was supported in part by NSF grants  DMS-1201480 and DMS-1501583.}

\begin{abstract}
This is an expository article on the averaged version of Colmez's conjecture, relating  Faltings heights of CM abelian varieties to Artin $L$-functions. 
It is based on the author's lectures at the Current Developments in Mathematics conference held at Harvard  in 2018.
\end{abstract}

\maketitle

\setcounter{tocdepth}{1}
\tableofcontents

\theoremstyle{plain}
\newtheorem{theorem}{Theorem}[subsection]
\newtheorem{bigtheorem}{Theorem}[section]
\newtheorem{proposition}[theorem]{Proposition}
\newtheorem{lemma}[theorem]{Lemma}
\newtheorem{corollary}[theorem]{Corollary}
\newtheorem{conjecture}[theorem]{Conjecture}

\theoremstyle{definition}
\newtheorem{definition}[theorem]{Definition}
\newtheorem{hypothesis}[theorem]{Hypothesis}

\theoremstyle{remark}
\newtheorem{remark}[theorem]{Remark}
\newtheorem{question}[theorem]{Question}
\newtheorem{example}[theorem]{Example}

\numberwithin{equation}{subsection}
\renewcommand{\thebigtheorem}{\Alph{bigtheorem}}


\section{Introduction}


The goal of this paper is to give an overview of a proof of the averaged Colmez conjecture, 
which relates the Faltings heights of CM abelian varieties to logarithmic derivatives of $L$-functions.


\subsection{Statement of the theorem}


Suppose $E$ is a CM field,   and  $A$ is an abelian variety defined over  $\Q^\alg \subset \C$  that 
 admits complex multiplication by the full ring of integers $\co_E \subset E$.  
 This means that $[E:\Q] = 2\cdot \mathrm{dim}(A)$, and  $A$ admits a ring homomorphism  $\co_E \to \End(A)$.

 It is a theorem of Colmez \cite{Colmez} that the Faltings height 
\[
h^\Falt_{(E,\Phi)} = h^\Falt(A)
\]
depends only on the CM type $(E,\Phi)$ of $A$, and not on the abelian variety $A$ itself. 
In the same paper in which he proved this theorem, Colmez proposed a conjectural formula for the value of the Faltings height as a linear combination of logarithmic derivatives of Artin $L$-functions at $s=0$.  
The precise statement is recalled here as Conjecture \ref{conj:general colmez}.

Colmez's conjectural formula generalizes both the Chowla-Selberg \cite{CS} formula for elliptic curves with complex multiplication, and results of Anderson \cite{And82} (following Deligne, Gross, and Shimura) on the periods of $A$ in the special case where $E/\Q$ is Galois with abelian Galois group. See also work of Maillot-Roessler \cite{MR} and Yoshida \cite{Yoshida}.
We will say nothing about Anderson's result and the work that preceded it, and instead refer the reader to the expository paper \cite{gross:expo}.  However, we will take a moment to state the Chowla-Selberg formula, so that the reader may compare it to Theorem \ref{thm:average colmez} below.

Chowla and Selberg originally stated their theorem as a formula for the values of Ramanujan's discriminant $\Delta$ at CM points in the complex upper half plane.  
It was noted in \cite{gross:thesis}, where the observation is credited to Deligne,  than one can use this to obtain a formula for the Faltings heights of elliptic curves with complex multiplication.  
The connection between the two formulations is explained in detail in \S \ref{s:chowla-selberg}.

When expressed in terms of Faltings heights, the Chowla-Selberg theorem says that if  $E$ is a quadratic imaginary field of discriminant $D_E$, and if $\Phi \subset \Hom(E,\C)$ is either of the two CM types of $E$,  then 
\begin{equation}\label{new cs}
h^\Falt_{(E,\Phi)}      =   - \frac{1}{2}  \frac{L'(0, \chi)}{ L(0,\chi)} -  \frac{1}{4} \log|D_E| -  \frac{1}{2} \log (   2\pi   )  .
\end{equation}
Here $\chi : \A^\times \to \{ \pm 1\}$ is the quadratic idele class character associated to the extension $E/\Q$, and $L(s,\chi)$ is the usual Dirichlet $L$-function (excluding the archimedean Euler factors).

Returning to the case of a general CM field, if one holds $E$ fixed and averages both sides of Colmez's conjectural formula over all CM types $\Phi \subset \Hom(E,\C)$, the result is the equality stated below as Theorem \ref{thm:average colmez}.  It is this theorem, which was proved  simultaneously by two groups of authors using different methods, that is the subject of this paper.

\begin{bigtheorem}[Andreatta-Goren-H-Madapusi Pera, Yuan-Zhang]\label{thm:average colmez}
If  $E$ is a CM field with maximal totally real subfield $F$ of degree $d=[F:\Q]$, then
\[
\frac{1}{2^d} \sum_\Phi h^\Falt_{(E,\Phi)} = 
-  \frac{1}{2}  \frac{ L'(0,\chi)  }{ L(0,\chi) }  -   \frac{1}{4}  \log \left| \frac{  D_{E}  }{  D_{F} }\right| 
-  \frac{ d   \log(2\pi)}{ 2 }  .
\]
Here  $D_E$ and $D_F$ are the (absolute) discriminants of $E$ and $F$, 
\[
\chi : \A_F^\times \to \{\pm 1\}
\] 
is the idele class character associated with the extension $E/F$, and the summation on the left is over the $2^d$ distinct CM types $\Phi \subset \Hom(E,\C)$.
\end{bigtheorem}

For applications of the theorem to the Andr\'e-Oort conjecture, see work of  Tsimermann \cite{Tsimerman}.

The  proof of Theorem \ref{thm:average colmez} given in \cite{AGHMP-2} is  based on  the method of Yang \cite{Ya10a,Ya10b}, 
who proved  Colmez's conjecture for  some non-Galois quartic CM fields by combining the theory of Borcherds products with some explicit calculations of arithmetic intersection multiplicities  on the integral model of a Hilbert modular surface.
The arguments of  \cite{AGHMP-2}  proceed along the same general lines, but with the Hilbert modular surface replaced by an orthogonal Shimura variety.

These arithmetic intersection calculations are much in the spirit of theorems and conjectures of Kudla \cite{KuAnnals},  and  variants of it explored in the work of Kudla-Rapoport \cite{KR99, KR00},  Bruinier-Yang \cite{BY09},  the author \cite{Ho12,Ho15}, and  various subsets of  those just named \cite{HY, BHY, BKY, BHKRY}.
In particular, \cite{Ho12} and \cite{BHKRY} contain  similar calculations, but for unitary rather than orthogonal Shimura varieties.  
The calculations are also reminiscent of those of Gross-Zagier \cite{GZ} and Gross-Keating \cite{GK}.

If  $a,b\in \C$, let us write \[ a\circeq b \] to mean that  $a-b$ is a $\Q$-linear combination of $\{ \log(p) : p \mbox{ prime}\}$.
In this paper we will outline the proof of Theorem \ref{thm:average colmez}, while also providing  a more-or-less complete proof of the weaker relation
\begin{equation}\label{periods}
\frac{1}{2^d} \sum_\Phi h^\Falt_{(E,\Phi)} \circeq 
-  \frac{1}{2}  \frac{ L'(0,\chi)  }{ L(0,\chi) }  -   \frac{1}{4}  \log \left| \frac{  D_{E}  }{  D_{F} }\right| 
-  \frac{ d   \log(2\pi)}{ 2 }  .
\end{equation}
The qualifier ``more-or-less'' means that the reader must accept as a black box the existence of Borcherds products, and a result of Bruinier-Kudla-Yang on their values at CM points of orthogonal Shimura varieties.

The point is that  the arithmetic intersection multiplicities appearing in the proof of Theorem \ref{thm:average colmez} decompose as sums of local terms, and the weaker relation  (\ref{periods}) only requires computing the archimedean contribution.  In particular, one only needs  the canonical models over $\Q$ of orthogonal Shimura varieties, not their integral models.


\subsection{Outline of the paper}


We now give a fairly detailed outline of both the structure of the paper, and of the proof of Theorem  \ref{thm:average colmez}   up to $\circeq$.

In \S \ref{s:chowla-selberg} we recall the original  Chowla-Selberg formula on the  CM values of Ramanujan's discriminant $\Delta$, and explain how to deduce from it the reformulation (\ref{new cs}) in terms of Faltings heights.  

Over the open modular curve $\mathcal{M}$ parametrizing elliptic curves over arbitrary schemes,  there is a metrized line bundle of weight one modular forms $\widehat{\taut}$. 
If  $A$ is an elliptic curve over a number field $\kk$ with everywhere good reduction, the theory of Neron models provides an extension of $A$ to an elliptic curve over $\co_\kk$, which defines a morphism
$
 \Spec(\co_\kk)  \to \mathcal{M}.
$
The Faltings height of $A$ can be computed by pulling back $\widehat{\taut}$ to a metrized line bundle on $ \Spec(\co_\kk)$ and taking its arithmetic degree.

In this way we see that the Faltings heights of elliptic curves are encoded by the metrized  line bundle $\widehat{\taut}$.  Ramanujan's discriminant enters the picture because it provides a concrete  trivialization of 
the $12^\mathrm{th}$ power of this line bundle.  

In \S \ref{s:colmez} we recall some  basic properties of abelian varieties with complex multiplication, and state Colmez's conjecture in full generality.  
There is  only one new idea here.  
Given a CM field $E$ and an embedding $E \to \C$, we construct the \emph{total reflex pair} $(E^\sharp,\Phi^\sharp)$, which  satisfies 
\begin{equation}\label{funny falt}
h^\Falt_{(E^\sharp,\Phi^\sharp)} = \frac{1}{2d} \sum_\Phi h^\Falt_{(E,\Phi)}.
\end{equation}
Using this, the averaged  Colmez conjecture for $E$ can be reduced to the \emph{exact} Colmez conjecture for $(E^\sharp,\Phi^\sharp)$.    
Admittedly,  at the moment this looks less like a reduction and more like   retrograde motion.

In \S \ref{s:orthogonal shimura variety} we introduce the orthogonal Shimura variety $M$ associated to a rational quadratic space $(V,Q)$   of signature $(n,2)$ and a choice of maximal lattice $L\subset V$ 
(\emph{maximal} means that  the quadratic form is $\Z$-valued on $L$, and that $L$ is maximal among all lattices with this property).
  It is a smooth $n$-dimensional Deligne-Mumford stack over $\Q$.

The Shimura variety $M$  is of Hodge type, but is  of PEL type only when $n\in \{ 0,1,2,3,4,6\}$.  Thus, while  $M$ does not have a simple moduli-theoretic interpretation, it can be embedded into a Siegel moduli space parametrizing polarized abelian varieties.
The Kuga-Satake construction provides a particularly natural way to do this, and pulling back the universal polarized abelian variety from the Siegel space yields the \emph{Kuga-Satake abelian scheme} \[A\to M\] of relative dimension $\dim(A)=2^{n+1}$.
The  Shimura variety $M$ also carries over it a line bundle of weight one modular forms $\taut$, and a family of special divisors $Z(m,\mu)$ indexed by  rational numbers $m>0$ and cosets $\mu\in L^\vee/ L$.

In \S \ref{ss:ks heights} we explain the connection between $\taut$ and Faltings heights.
If $y\in M(\kk)$ is a point valued in a number field $\kk$, the Faltings height of the fiber $A_y$ satisfies
\begin{equation}\label{ks faltings}
 \frac{-1}{ [\kk : \Q ] }\sum_{ \sigma : \kk \to \C}  \log ||     v^\sigma     ||   \circeq       \frac{1}{ 2^n }  \cdot h^\Falt ( A_y) +    \log(2\pi)  
\end{equation}
for any nonzero vector $v \in \taut_y$.  Here  $||  \cdot || $ is the  metric (\ref{pet metric}) on $\taut$.

In \S  \ref{s:integral}  we define a flat extension of $M$ to a stack  $\mathcal{M}$ over $\Z$.
We also extend the Kuga-Satake abelian scheme to $\mathcal{A} \to \mathcal{M}$, and construct extensions $\taut$ and $\mathcal{Z}(m,\mu)$  of the  line  bundle of modular forms and the family of special divisors to $\mathcal{M}$.
In particular, we obtain a class
\[
\widehat{\taut} \in  \widehat{\Pic}(\mathcal{M})
\]
in the group of metrized line bundles on $\mathcal{M}$.

The Shimura variety  $\mathcal{M}$, endowed with its line bundle $\taut$ and its Kuga-Satake family, is analogous to the modular curve endowed with its line bundle of modular forms and universal elliptic curve.
However, on $\mathcal{M}$   there is no natural analogue of the discriminant $\Delta$ that can be used to trivialize a power of $\taut$.

This is addressed in \S \ref{s:harmonic}, where we recall the theory of Borcherds products. 
The metaplectic double cover of  $\SL_2(\Z)$ acts via the Weil representation $\weil_L$ on the space $S_L$ of $\C$-valued functions on $L^\vee /L$.  Suppose
\[
f(\tau)  = \sum_{ m\gg -\infty} c(m) \cdot q^m \in M^!_{1-\frac{n}{2}}( \weil_L)
\]
is a weakly holomorphic form of weight $1-n/2$ and representation $\weil_L$.  
Each Fourier coefficient $c(m)\in S_L$ decomposes as a linear combination 
\[
c(m) = \sum_{ \mu \in L^\vee / L } c(m,\mu) \cdot \varphi_\mu,
\] 
where $\varphi_\mu \in S_L$ is the characteristic function of $\mu$.

Assuming that all $c(m,\mu)$ are integers,  the Borcherds product of $f$ is a  rational section $\bm{\psi}(f)$ of $\taut^{\otimes c(0,0)}$ with divisor 
 \[
 \mathcal{Z}(f) = \sum_{ \substack{ m >0 \\  \mu \in L^\vee / L  }}  c(-m,\mu) \cdot \mathcal{Z}(m,\mu).
 \]
The theory of regularized theta lifts allows one to construct a Green function $\Phi(f)$ for $\mathcal{Z}(f)$, and   the resulting class 
\[
\widehat{\mathcal{Z}}(f) = ( \mathcal{Z}(f) ,\Phi(f)) \in \widehat{\mathrm{CH}}^1(\mathcal{M})
\]
in the codimension one arithmetic Chow group of $\mathcal{M}$ is 
essentially equal to  $\widehat{\taut}^{\otimes c(0,0)}$ under the isomorphism
\begin{equation}\label{space of lines}
\widehat{\mathrm{CH}}^1(\mathcal{M}) \iso \widehat{\Pic}(\mathcal{M}).
\end{equation}
The precise relation is found in Corollary \ref{cor:taut to special}.

In \S \ref{s:big cm} we choose the quadratic space  $(V,Q)$  in a particular way.  Let $E$ be a CM field with 
 $[E:\Q]=2d>2$.  Choose a $\xi \in F^\times$  that is negative at some fixed real embedding $F\to \R$, and positive at all remaining embeddings.   
 Endow $V=E$, viewed as a vector space over $\Q$ of dimension $2d$, with the   quadratic form
\[
Q(x) = \mathrm{Trace}_{F/\Q} ( \xi  x\overline{x} )
\]
of signature $(2d-2,2)$.

The Shimura variety associated to this quadratic space comes equipped with distinguished cycle 
$
\mathcal{Y}_\Z \to\mathcal{M},
$
which is regular and finite  flat over $\Spec(\Z)$.  
In particular it has dimension $1$.  We obtain a linear functional
\[
[ - : \mathcal{Y}_\Z]  : \widehat{\mathrm{CH}}^1(\mathcal{M}) \to \R
\]
as the composition
\[
\widehat{\mathrm{CH}}^1(\mathcal{M}) \to   \widehat{\mathrm{CH}}^1(\mathcal{Y}_\Z) \map{\widehat{\deg}} \R,
\]
where the first arrow is pullback of arithmetic divisors, and the second is the arithmetic degree.

The Kuga-Satake abelian scheme $\mathcal{A} \to \mathcal{M}$ acquires complex multiplication of a particular type when restricted to $\mathcal{Y}_\Z$.  
More precisely, for every  point $y\in \mathcal{Y}_\Z(\C)$ the fiber $\mathcal{A}_y$ is isogenous to a power of an abelian variety with complex multiplication by $E^\sharp$, and CM type a Galois conjugate of $\Phi^\sharp$. 
Using  this and the connection between  $\widehat{\taut}$ and the Faltings height provided by (\ref{ks faltings}) we deduce 
\[
\frac{  [ \widehat{\taut} : \mathcal{Y}_\Z ] }{ \deg_\C(\mathcal{Y}_\Z) }
 \circeq 
 \frac{ 1}{ 2^{d-2}   }    h^\Falt_{( E^\sharp,\Phi^\sharp) }  +      \log(2\pi)  ,
\]
which we rewrite, using  (\ref{funny falt}), as
\begin{equation}\label{intro taut height}
\frac{  [ \widehat{\taut} : \mathcal{Y}_\Z ] }{ \deg_\C(\mathcal{Y}_\Z) } \circeq 
 \frac{ 1}{ d\cdot 2^{d-1}   }    \sum_\Phi h^\Falt_{(E,\Phi)}  +      \log(2\pi)  .
\end{equation}
See  Theorem \ref{thm:taut to faltings} for a stronger statement.

In \S \ref{s:special to eisenstein} we turn to the calculation of the arithmetic intersection
$[ \widehat{\mathcal{Z}}(f) : \mathcal{Y}_\Z ]$.  The cycles in question intersect properly, and so this intersection decomposes as the sum 
\[
[ \widehat{\mathcal{Z}}(f) : \mathcal{Y}_\Z ] = [ \mathcal{Z}(f) : \mathcal{Y}_\Z ]_\mathrm{fin} + [ \Phi(f) : \mathcal{Y}_\Z ]_\infty
\]
of a contribution from finite places  and an archimedean term

The archimedean contribution was computed by Bruinier-Kudla-Yang \cite{BKY}, who showed that it can be expressed in terms of the coefficients of the central derivative of a Hilbert modular Eisenstein series of weight one. 
 If one works modulo $\circeq$,   all of the coefficients vanish except for the constant term, and the constant term is essentially 
the logarithmic derivative of $L(s,\chi)$ at $s=0$.   In particular, the theorem of Bruinier-Kudla-Yang shows that 
\[
\frac{ 2d }{  c(0,0)   } \cdot \frac{ [ \Phi(f) : \mathcal{Y}_\Z ]_\infty }{  \deg_\C( \mathcal{Y}_\Z)  }
\circeq
- 2  \frac{  L'(0,\chi)  }{ L(0,\chi) }  - \log \left| \frac{  D_{E}  }{  D_{F} }\right|  +d  \cdot  \log(4\pi e^\gamma) .
\]
As $[ \mathcal{Z}(f) : \mathcal{Y}_\Z ]_\mathrm{fin}  \circeq 0$, we deduce 
\begin{equation}\label{intro intersect}
\frac{ 2d }{  c(0,0)   } \cdot \frac{ [ \widehat{\mathcal{Z}}(f)  : \mathcal{Y}_\Z ] }{  \deg_\C( \mathcal{Y}_\Z)  }
\circeq
- 2  \frac{  L'(0,\chi)  }{ L(0,\chi) }  - \log \left| \frac{  D_{E}  }{  D_{F} }\right|  +d  \cdot  \log(4\pi e^\gamma) .
\end{equation}
See Corollary \ref{cor:intersection} for a stronger statement.

Combining  (\ref{intro taut height}),  (\ref{intro intersect}), and the relation 
 \[
\frac{d}{2} \cdot   \frac{  [  \widehat{\taut}   :    \mathcal{Y}_\Z  ]  }  {  \deg_\C( \mathcal{Y}_\Z)      }    
 =\frac{d}{2c(0,0)} \cdot  \frac{  [ \widehat{\mathcal{Z}}(f) : \mathcal{Y}_\Z ] }  {   \deg_\C( \mathcal{Y}_\Z)  }  
- \frac{d}{4} \cdot  \log ( 4 \pi  e^\gamma )  
 \]
 obtained by comparing $\widehat{\taut}$ and $\widehat{\mathcal{Z}}(f)$ under the isomorphism (\ref{space of lines}),
we find that the equality in Theorem \ref{thm:average colmez} holds up to $\circeq$.  

In order to upgrade from $\circeq$ to actual equality, one must strengthen both (\ref{intro taut height}) and (\ref{intro intersect}).
It is the latter which is more difficult, and occupies much of  \cite{AGHMP-2}.
To do this one must  compute  the finite intersection multiplicities 
$[\mathcal{Z}(m,\mu) : \mathcal{Y}_\Z]_\mathrm{fin}$ of all special divisors, and compare them with the same Fourier coefficients of Eisenstein series appearing in the work of Bruinier-Kudla-Yang.  The precise statement is Theorem \ref{thm:deformation}.


\section{The Chowla-Selberg formula}
\label{s:chowla-selberg}


We recall here the original statement of the Chowla-Selberg formula concerning CM values of Ramanujan's discriminant, and explain how to deduce from it the reformulation  in terms of Faltings heights.   
To some extent this is a formality, and our real purpose is to acquaint the  reader some ideas that will appear  in our sketch of the proof of Theorem \ref{thm:average colmez}. 

The central idea that we hope to convey here is that Faltings heights of elliptic curves can be computed using the metrized line bundle of weight one modular forms on the modular curve.  
As Ramanujan's discriminant provides a trivialization of the $12^\mathrm{th}$ power of this line bundle, we obtain a connection between discriminants and Faltings heights.


\subsection{The analytic formulation}


As usual, we let the group $\SL_2(\Z)$ act on the complex upper half-plane via
\[
g \cdot  \tau =  \frac{ a \tau +b}{c\tau+d } ,
\]
where  $g=\left(\begin{smallmatrix}  a & b \\ c & d \end{smallmatrix}\right) $.
By a \emph{weak modular form} of weight $k$ we mean a holomorphic function $f : \mathcal{H} \to\C$ on the complex upper half-plane
satisfying the transformation law
\[
f (g\tau) = (c\tau+d)^k   \cdot  f(\tau)
\]
for all $g\in \SL_2(\Z)$ as above.  The adjective \emph{weak} is added because we are imposing no conditions  on the behavior of $f(\tau)$ as $\tau \to i \infty$.

The most famous example of such a function is Ramanujan's modular discriminant of weight $12$,  defined by 
\[
\Delta(\tau) =  q \prod_{n=1}^\infty(1-q^n)^{24}
\]
where  $q=e^{2\pi i \tau}$.
We   may view $\Delta$ as a function on lattices $L \subset \C$ in the usual way:  choose a $\Z$-basis  $ \omega_1,\omega_2 \in L$ such  that $\omega_1/\omega_2\in \mathcal{H}$, and set 
\[
\Delta(L) =  \omega_2^{-12} \cdot \Delta ( \omega_1/\omega_2) .
\]
It is clear from the transformation law that this doesn't depend on the choice of basis $\{ \omega_1,\omega_2\}$, and of course
\[
\Delta(\alpha L) = \alpha^{-12} \Delta(L)
\]
for any $\alpha\in \C^\times$.

Fix a quadratic imaginary field $E=\Q(\sqrt{-d})$ of discriminant $-d$.   
Denote by  $\mathrm{CL}(E)$ the ideal class group of $E$, by $h=|\mathrm{CL}(E)|$ its  class number, and by $w=|\co_E^\times|$ the  number of number of roots of unity in $E$.  Let 
\[
\chi : (\Z/ d\Z)^\times \to \{\pm 1\}
\]
 be the corresponding  Dirichlet character.

Choose  an embedding $E \subset \C$. 
This allows us to view  a fractional ideal $\mathfrak{a} \subset E$  as a  lattice in $\C$, and the product 
$\Delta(\mathfrak{a})\Delta(\mathfrak{a}^{-1})$   depends only on the image of $\mathfrak{a}$ in the ideal class group.

\begin{theorem}[Chowla-Selberg  \cite{CS}, see also \cite{Weil}]
We have
\[
 (2\pi d)^{12 h}  \prod_{ \mathfrak{a} \in \mathrm{CL}(E) } \Delta(\mathfrak{a}) \Delta( \mathfrak{a}^{-1} ) 
=  \prod_{  0 < a < d } \Gamma \left(a/d \right)^{6 w  \chi (a) }.
\]
\end{theorem}

 Using Lerch's formula 
\[
\log(d) + \frac{L'(0, \chi)}{L(0,\chi)}   =   \frac{ w }{ 2 h}  \sum_{ 0 < a < d }   \chi(a)  \log \big( \Gamma ( a/ d) \big) 
\]
 one can rewrite the Chowla-Selberg formula as
 \begin{equation}\label{better CS}
\frac{1}{24 h} \sum_{ \mathfrak{a} \in \mathrm{CL}(E) } \log|  \Delta(\mathfrak{a}) \Delta( \mathfrak{a}^{-1} )  |
  =    \frac{1}{2}  \frac{L'(0, \chi)}{L(0,\chi)}  - \frac{1}{2}  \log(2\pi).
  \end{equation}


\subsection{Arithmetic intersections}
\label{ss:arakelov}


We need some rudimentary ideas from arithmetic intersection theory, following \cite{GS,soule,gillet}.
Let  $\mathcal{M}$ be a locally integral Deligne-Mumford stack, flat and of finite type over $\Spec(\Z)$, and with smooth generic fiber.

\begin{definition}
Suppose $\mathcal{Z}$ is a Cartier divisor on $\mathcal{M}$.  A \emph{Green function} for $\mathcal{Z}$ is a smooth real-valued function $\Phi$ on the complex orbifold $\mathcal{M}(\C) \smallsetminus \mathcal{Z}(\C)$ satisfying the following two properties.
\begin{enumerate}
\item
If  $f$ is a meromorphic function on  a holomorphic  orbifold chart $U\to  \mathcal{M}(\C)$ satisfying
\[
\mathrm{div}(f) = \mathcal{Z}(\C)|_U ,
\]
 then  $\Phi|_U  + 2 \log | f |$, initially defined on $U \smallsetminus \mathcal{Z}(\C)|_U$,   extends  to a smooth function on $U$;
\item
Pullback by complex conjugation on $\mathcal{M}(\C)$ fixes $\Phi$.
\end{enumerate}
\end{definition}

\begin{definition}
An \emph{arithmetic divisor} on $\mathcal{M}$ is a pair 
$
\widehat{\mathcal{Z}} = (\mathcal{Z}, \Phi)
$ 
consisting of a Cartier divisor $\mathcal{Z}$ on $\mathcal{M}$ and a Green function $\Phi$ for $\mathcal{Z}$.
An arithmetic divisor as above is \emph{principal} if it has the form
\[
\widehat{\mathcal{Z}} = ( \mathrm{div}(f) , -2 \log |f| )
\]
for some rational function $f$ on $\mathcal{M}$.
\end{definition}

\begin{definition}
A \emph{metrized line bundle} on $\mathcal{M}$ is a pair 
\[
\widehat{\mathcal{L}} = (\mathcal{L} , || \cdot || )
\]
consisting of a line bundle on $\mathcal{M}$, and a smoothly varying family of Hermitian metrics on its complex fiber. 
We further require the metrics to be invariant under  pullback by complex conjugation on $\mathcal{M}(\C)$.
\end{definition}

The \emph{codimension one arithmetic Chow group} $\widehat{\mathrm{CH}}^1(\mathcal{M})$ is the quotient of the group of all arithmetic divisors on $\mathcal{M}$ by the subgroup of principal arithmetic divisors.  If we denote by $\widehat{\Pic}(\mathcal{M})$ the group of metrized line bundles (under tensor product), there is a canonical isomorphism
\begin{equation}\label{metric to chow}
\widehat{\Pic}(\mathcal{M}) \iso \widehat{\mathrm{CH}}^1(\mathcal{M})
\end{equation}
sending a metrized line bundle $\widehat{\mathcal{L}}$ to the arithmetic divisor 
\[
\widehat{\mathrm{div}}(s) = (\mathrm{div}(s) , - 2  \log||s||)
\] 
for any nonzero rational section $s$ of $\mathcal{L}$.

Now suppose that $\mathcal{Y}$ is a regular Deligne-Mumford stack, finite and flat over $\Spec(\Z)$.   
In particular  $\dim(\mathcal{Y})=1$, and the assumption of regularity means that we need not distinguish between Cartier divisors and Weil divisors.  
There is a linear functional 
\[
\widehat{\deg} : \widehat{\mathrm{CH}}^1(\mathcal{Y}) \to \R
\] 
called the \emph{arithmetic degree}, defined as follows.  As any  divisor on $\mathcal{Y}$ has empty generic fiber,  any arithmetic divisor on $\mathcal{Y}$ decomposes uniquely as a finite sum
\[
  ( 0 ,\Phi )  +   \sum_i m_i  \cdot  (\mathcal{Z}_i , 0 ) 
\]
in which $m_i\in \Z$,  each $\mathcal{Z}_i$ is an irreducible effective  divisor on $\mathcal{Y}$ supported in a single nonzero characteristic, and $\Phi$ is a complex conjugation invariant function on the $0$-dimensional orbifold $\mathcal{Y}(\C)$.  
Thus it suffices to define
\[
\widehat{\deg}( \mathcal{Z},0) = \sum_{ y\in \mathcal{Z}(\F_p^\alg) } \frac{\log(p)}{ |\Aut(y) | }
\]
when $\mathcal{Z}$ is irreducible and supported in characteristic $p$, and define
\[
\widehat{\deg}( 0 , \Phi ) =  \frac{1}{2} \sum_{ y \in \mathcal{Y}(\C) } \frac{\Phi(y) }{ |\Aut(y) | }.
\]

Keeping $\mathcal{Y}$ and $\mathcal{M}$ as above, suppose we are given a morphism
\[
\mathcal{Y} \to \mathcal{M}.
\]
Composing the arithmetic degree with  pullback of metrized line bundles
\[
\widehat{\mathrm{CH}}^1(\mathcal{M})\iso \widehat{\Pic}(\mathcal{M}) \to \widehat{\Pic}(\mathcal{Y}) \iso \widehat{\mathrm{CH}}^1(\mathcal{Y})
\]
defines a linear functional
\[
 [- : \mathcal{Y} ] : \widehat{\mathrm{CH}}^1(\mathcal{M}) \to \R
\]
called \emph{arithmetic intersection against $\mathcal{Y}$}.

Suppose $\widehat{\mathcal{Z}} = ( \mathcal{Z} ,\Phi)$ is an arithmetic divisor on $\mathcal{M}$, and suppose further that the underlying Cartier divisor $\mathcal{Z}$ is effective and meets $\mathcal{Y}$ properly
 in the sense that 
\[
\mathcal{Z} \cap \mathcal{Y} \define \mathcal{Z} \times_{\mathcal{M}} \mathcal{Y}
\]
has dimension $0$.   In this simple case there is a decomposition 
\[
[ \widehat{\mathcal{Z}} : \mathcal{Y} ] 
= 
 [  \mathcal{Z} : \mathcal{Y} ]_\mathrm{fin} + [ \Phi : \mathcal{Y}]_\infty
\]
where the finite contribution is 
\begin{equation}\label{finite decomp by points}
[  \mathcal{Z} : \mathcal{Y} ]_\mathrm{fin}=
\sum_p  \log(p) \left(
\sum_{   y\in (\mathcal{Z}\cap \mathcal{Y})( \F_p^\alg)     }
 \frac{ \mathrm{length}\big( \co^\mathrm{et}_{ \mathcal{Z}\cap \mathcal{Y}  ,y  }  \big)  }{ | \Aut(y) | } \right),
\end{equation}
and the archimedean contribution is
\[
[ \Phi : \mathcal{Y} ]_\infty
= \frac{1}{2} \sum_{ y \in \mathcal{Y}(\C) } \frac{\Phi(y) }{ |\Aut(y) | } .
\]


\subsection{The Faltings height}


Suppose $A$ is an elliptic curve defined over a number field  $\kk$, and that $A$ has everywhere good reduction.  
Denote in the same way its   N\'eron model 
\begin{equation}\label{elliptic neron}
 A \to \Spec(\co_\kk).
\end{equation}

The space of global $1$-forms  on $A$ is a projective $\co_\kk$-module of rank one.  
After choosing a nonzero element  
\[
\eta \in H^0 ( A ,\Omega_{A/\co_\kk}),
\]
define the archimedean part of the Faltings height by
\[
h^\Falt _\infty ( A, \eta ) = \frac{-1}{ 2 [\kk:\Q] }  \sum_{ \sigma : \co_\kk \to \C}
\log \big|   \int_{ A^\sigma(\C) } \eta^\sigma \wedge \overline{\eta^\sigma}\,  \big|,
\]
where $A^\sigma \to \Spec(\C)$  is the base change of $A$ by $\sigma$, and  similarly for $\eta^\sigma$.
Define the finite part of the Faltings height by
\[
h^\Falt_f(A,\eta) = \frac{1}{  [\kk:\Q] }    \log|   H^0(A, \Omega_{A/\co_\kk } )   /  \co_\kk \eta     | .
\]
The \emph{Faltings height}
\[
h^\Falt(A) =    h^\Falt _\infty ( A,\eta)   + h^\Falt_f(A, \eta) 
\]
is independent of the choice of  $\eta$, and is unchanged if we enlarge $\kk$.

Let $\mathcal{M}$ be the Deligne-Mumford stack over $\Z$ classifying elliptic curves,  and let $\pi : \mathcal{A} \to \mathcal{M}$ be the universal elliptic curve.  The \emph{line bundle of weight one modular forms} on $\mathcal{M}$ is  
\begin{equation*}
\taut  =   \pi_* \Omega_{\mathcal{A} / \mathcal{M} } \iso \Lie( \mathcal{A})^{-1} .
\end{equation*}
We endow  $\taut$ with the Faltings metric, which at a  complex point $y\in \mathcal{M}(\C)$ assigns the norm
\[
|| \eta ||^2_y =  \Big|   \int_{\mathcal{A}_y(\C) } \eta \wedge \overline{\eta}   \, \Big|
\]
to a global $1$-form
$
\eta \in \taut_y \iso H^0( \mathcal{A}_y , \Omega_{ \mathcal{A}_y / \C}),
$
and denote by 
\[
\widehat{\taut} \in \widehat{\Pic}(\mathcal{M})
\]
 the resulting metrized line bundle on $\mathcal{M}$.

The  N\'eron model (\ref{elliptic neron}) is  an elliptic curve over $\co_\kk$, and so determines a morphism
\[
\mathcal{Y}\define \Spec(\co_\kk) \to \mathcal{M}.
\]
Unwinding the  definitions, we find that 
\begin{equation}\label{elliptic faltings}
h^\Falt(A) =  \frac{ [ \widehat{\taut}  : \mathcal{Y} ] } { [\kk :\Q]  }.
\end{equation}


\subsection{Chowla-Selberg revisited}


 For each integer $k$ there is  an isomorphism of $\C$-vector spaces
\begin{equation*}
\xymatrix{
{   \{ \mbox{weak modular forms of weight }k \}  }  \ar[d]^{ f\mapsto \underline{f} }   \\ 
{   \{ \mbox{global holomorphic sections of } \taut^{\otimes k} \} .  }
}
\end{equation*}
To make this isomorphism explicit, let the group  $\SL_2(\Z)$ act on  $\mathcal{H} \times \C$ by 
 \[
 \left(\begin{matrix} a & b\\ c & d \end{matrix}\right)   \cdot  ( \tau, z )     = \left(  \frac{a\tau + b}{c \tau +d}  , (c\tau+d)^k \cdot z   \right).
\]
There is a commutative diagram
\[
 \xymatrix{
  {   \SL_2(\Z) \backslash (\mathcal{H} \times \C)   } \ar[r]  \ar[d]_{\iso}  &       {  \SL_2(\Z) \backslash \mathcal{H}  } \ar[d]_{\iso} \\
  {  \taut^{ \otimes k} (\C)  } \ar[r]   &  {\mathcal{M}(\C)    } 
 }
\]
in which the  isomorphism on the left  sends   $( \tau , z) \in \mathcal{H}\times \C$ to   the complex elliptic curve \[A_\tau(\C)= \C/ (\Z\tau + \Z)\] together with the vector 
 $
 z\in \C \iso  \Lie(A_\tau)^{ \otimes -k} 
 $
 in the fiber of $\taut^{\otimes k}$ at this point.
 Here the Lie algebra  $\Lie(A_\tau)$, and hence also its tensor powers,  has been  trivialized in the obvious way.
 If $f$ is a weak modular form of weight $k$, then 
\[
\underline{f}(\tau) = \big(   \tau , (2\pi i )^k  f(\tau) \big)
\]
defines a section   to the top horizontal arrow, which we interpret as a global holomorphic section of $\taut^{\otimes k}$.

Using the $q$-expansion principle, one can show that the holomorphic global section $\underline{\Delta}$ on $\mathcal{M}(\C)$ determined by $\Delta$  is algebraic and descends to the $\Q$-stack $\mathcal{M}_\Q$.  In fact, it  extends (necessarily uniquely) to a global section 
\[
\underline{\Delta} \in H^0 ( \mathcal{M} ,\taut^{\otimes 12} ),
\]
which trivializes the line bundle $\taut^{\otimes 12}$.

\begin{proposition}\label{prop:disc trivializes}
The isomorphism (\ref{metric to chow}) sends the metrized line bundle $\widehat{\taut}^{\otimes 12}$ to the arithmetic divisor
\[
(0 ,\Phi) \in \widehat{\mathrm{CH}}^1(\mathcal{M}),
\]
where, for any  $s \in \mathcal{M}(\C)$ we choose a lattice $L_s\subset \C$ satisfying 
\[
\mathcal{A}_s(\C) \iso \C / L_s,
\]
and  define 
\[
 \Phi(s)   =  -12 \cdot \log\Big|  4\pi^2    \Delta(L_s)^{1/6}     \int_{\C / L_s } dz\wedge d\overline{z}  \,  \Big|.
\]
\end{proposition}

\begin{proof}
As $\underline{\Delta}$ defines a nowhere vanishing section of $\taut^{\otimes 12}$, 
 the isomorphism (\ref{metric to chow}) sends 
\[
\widehat{\taut}^{\otimes 12} \mapsto \big(0 , - 2\log|| \underline{\Delta} || \big).
\]
At any point $s\in \mathcal{M}(\C)$ the fiber
\[
\underline{\Delta}_s  \in \taut_s^{\otimes 12}  \iso H^0( \mathcal{A}_s(\C) ,\Omega_{\mathcal{A}_s/\C }^{\otimes 12} )
\]
is given explicitly by  
$
 \underline{\Delta}_s = \Delta(L_s) \cdot  ( 2\pi i \, dz)^{\otimes 12} ,
$
and so 
\[
|| \underline{\Delta}||_s   = (2\pi )^{12} \cdot    | \Delta(L_s) |   \cdot \Big| \int_{\C / L_s } dz\wedge d\overline{z} \, \Big|^{6}.
\]
The claim follows immediately.
\end{proof}

We can use the above proposition to give a formula for the Faltings height of any elliptic curve $A$ over a number field $\kk \subset \C$ with everywhere good reduction.   Let $s\in \mathcal{M}(\kk)$ be the corresponding point on the modular curve.  The assumption of good reduction implies that the corresponding morphism
$
\Spec(\kk) \to \mathcal{M}
$
extends to 
\[
\mathcal{Y} \define \Spec(\co_\kk) \to  \mathcal{M},
\]
and combining (\ref{elliptic faltings}) with Proposition \ref{prop:disc trivializes} gives
\begin{align}
h^\Falt(A) 
& = \frac{1}{12} \cdot  \frac{  [  \widehat{\taut}^{\otimes 12} : \mathcal{Y} ]  }{ [ \kk : \Q ] }  \nonumber  \\
& = \frac{1}{24 [\kk : \Q]} \sum_{ y\in \mathcal{Y}(\C) } \frac{ \Phi(y) }{ |\Aut(y) | }   \nonumber   \\ 
& =   \frac{-1}{ 2  [\kk : \Q] } 
  \sum_{\sigma :  \kk \to \C} \log  \Big|  4\pi^2 \Delta(L_{s^\sigma}) ^{\frac{1}{6} }   \int_{\C / L_{s^\sigma} } dz\wedge d\overline{z} \, \Big| .\label{simple faltings}
\end{align}
In the case where $A$ has complex multiplication, this formula simplifies even further.

\begin{proposition}
Let   $A$ be an elliptic curve defined over a number field $\kk$, and having everywhere good reduction.  
If  $A$ admits complex multiplication by $\co_E$, then
\[
 \frac{1}{24 h}  \sum_{ \mathfrak{a} \in \mathrm{CL}(E) }  \log \big|    \Delta(\mathfrak{a}) \Delta(\mathfrak{a}^{-1} )       \big| 
= 
-    h^\Falt ( A  )    -  \frac{1}{4} \log (d)  -\log(2\pi).
\]
\end{proposition}

\begin{proof}
After enlarging  $\kk$ we are free  to assume that $\kk/\Q$ is Galois.
A choice of embedding $\kk \hookrightarrow \C$ then  identifies 
\[
\Gal(\kk/\Q) \iso \{ \mbox{embeddings }\kk\to \C \}.
\] 
It follows from the theory of complex multiplication that  $\kk$ contains the Hilbert class field $H$ of $E$, and that 
$A\iso A^\sigma$ for every $\sigma \in \Gal(\kk/H)$.

Every  $\sigma\in \Gal(\kk/\Q)$ determines  a class $\mathfrak{a}_\sigma \in \mathrm{CL}(E)$ characterized by 
\[
A^\sigma(\C) \iso \C / \mathfrak{a}_\sigma.
\]
The theory of complex multiplication implies that the resulting function $\Gal(\kk/\Q) \to \mathrm{CL}(E)$ factors through a surjective two-to-one function
\[
\Gal( H/\Q) \to \mathrm{CL}(E).
\]
If $c\in \Gal(H/\Q)$ is complex conjugation,  then $\mathfrak{a}_{c\circ \sigma}$ is the complex conjugate of $\mathfrak{a}_\sigma$ in the ideal class group.  

We now fix, for every $\sigma \in \Gal(H/\Q)$ a fractional ideal $\mathfrak{a}_\sigma \subset E$ such that $\mathcal{A}^\sigma(\C) \iso \C / \mathfrak{a}_\sigma$,  and do this in such a way that  $\mathfrak{a}_{c\circ \sigma}$ is the complex conjugate of  $\mathfrak{a}_\sigma$ in $E$.
This implies 
\[
 \Delta( \mathfrak{a}_{c\circ \sigma} ) = \mathrm{N}(\mathfrak{a}_\sigma)^{-12}  \Delta( \mathfrak{a}^{-1}_{ \sigma} ),
\]
and the formula (\ref{simple faltings}) simplifies to 
\begin{align*}
h^\Falt ( A  ) 
&=  \frac{-1}{ 2[ H :\Q] }  \sum_{ \sigma \in \Gal(H/\Q)  }
\log \big|   4\pi^2  \Delta(\mathfrak{a}_\sigma)^{\frac{1}{6}} \int_{ \C   /  \mathfrak{a}_\sigma } dz \wedge d\overline{z}\,  \big| \\
&=  \frac{-1}{ 2[ H:\Q] }  \sum_{ \sigma \in \Gal(H/\Q) }
\log \big|  4\pi^2   d^{\frac{1}{2}}  \Delta(\mathfrak{a}_\sigma)^{\frac{1}{6}} \mathrm{N}(\mathfrak{a}_\sigma)    \big|  \\
&=  \frac{-1}{ [ H:\Q] }  \sum_{ \sigma \in \Gal(H/\kk) }
\log \big|   4\pi^2  d^{\frac{1}{2}}  \Delta(\mathfrak{a}_\sigma)^{\frac{1}{12}}  \Delta(\mathfrak{a}_{c\circ \sigma})^{\frac{1}{12}}    \mathrm{N}(\mathfrak{a}_\sigma)    \big|  \\
&=  \frac{-1}{2h}  \sum_{ \mathfrak{a} \in \mathrm{CL}(E) }
\log \big|   4\pi^2   d^{\frac{1}{2}}  \Delta(\mathfrak{a})^{\frac{1}{12}}  \Delta(\mathfrak{a}^{-1} )^{\frac{1}{12}}       \big| .\end{align*}
This is equivalent to the stated formula.
\end{proof}

Combining the preceding proposition with (\ref{better CS})  yields the Chowla-Selberg formula in the form (\ref{new cs}).


\section{Colmez's conjecture}
\label{s:colmez}


This section contains a general discussion of CM algebras and CM abelian varieties,  including the statement of Colmez's conjecture in full generality.  

There is only one  new idea.  
We will show in Corollary \ref{cor:reflex height} that the averaged Faltings height appearing in Theorem \ref{thm:average colmez} can be rewritten  as the Faltings height of a \emph{single} CM abelian variety with  a  very particular CM type $(E^\sharp,\Phi^\sharp)$.  
Such abelian varieties  will later appear in the Kuga-Satake family over an orthogonal Shimura variety.


\subsection{CM fields}
\label{ss:reflex}


Before discussing abelian varieties with complex multiplication, we recall some basic facts about CM fields  and CM types.
We also introduce  the notions  of total reflex algebra and total reflex type of a CM field.

\begin{definition}
The following terminology is standard:
\begin{enumerate}
\item
A \emph{CM field}   is a totally imaginary quadratic extension of a totally real number field.
\item
A \emph{CM algebra} is a finite product of CM fields.
\item
A \emph{CM type} of a CM algebra $E$ is a set
\[
\Phi \subset \Hom(E ,\C) 
\]
of $\Q$-algebra maps such that $\Phi \sqcup \overline{\Phi} = \Hom(E, \C)$.  
\item
A \emph{CM pair} is a pair $(E,\Phi)$ consisting of a CM algebra $E$ and a CM type $\Phi\subset \Hom(E,\C)$.
\end{enumerate}
\end{definition}

\begin{remark}
Suppose $E = E_1\times \cdots \times E_r$ is a CM algebra, with each factor $E_i$  a CM field.
There is a canonical identification 
\[
\Hom(E,\C)  = \Hom(E_1,\C)  \sqcup  \cdots \sqcup \Hom(E_r,\C) ,
\]
and every CM type of $E$ has the form 
\[
\Phi = \Phi_1 \sqcup \cdots \sqcup \Phi_r
\]
for $\Phi_i$  a CM type of $E_i$.
\end{remark}

Let $\Q^\alg \subset \C$ be the algebraic closure of $\Q$ in $\C$, and set $G_\Q = \Gal(\Q^\alg/\Q)$.
The following is an exercise in Galois theory.

\begin{proposition}\label{prop:Gsets}
The construction
\[
B \mapsto \Hom(B,\C)
\] 
establishes  a bijection  between the set of isomorphism classes of finite \'etale $\Q$-algebras,  and the set of isomorphism classes of 
finite sets with a continuous action of $G_\Q$.
\end{proposition}


Suppose  $E$ is a CM field, and choose an embedding  $\iota_0 : E \to \C$.  To this data we can associate a 
\emph{total reflex pair}  $(E^\sharp,\Phi^\sharp)$ as follows.  First note that   the group $G_\Q$ acts on the set of  all CM types of $E$ by
\[
g \circ \Phi = \{ g \circ \varphi : \varphi \in \Phi \}.
\]
By Proposition \ref{prop:Gsets} there is a finite \'etale $\Q$-algebra $E^\sharp$ characterized by
\[
\{ \mbox{CM types of $E$} \}  \iso \Hom(E^\sharp , \C )
\]
as finite sets with $G_\Q$-actions.  
 The embedding $\iota_0$  determines a subset
\[
\Phi^\sharp = \{ \mbox{CM types of $E$ containing $\iota_0$} \}   \subset   \{ \mbox{CM types of $E$} \} \iso \Hom(E^\sharp , \C ).
\]
We call $E^\sharp$ and $\Phi^\sharp$ the \emph{total reflex algebra} and  \emph{total reflex type}, respectively.

\begin{proposition}\label{prop:reflex pair}
Fix a CM field $E$ and an embedding  $\iota_0:E \to \C$.
\begin{enumerate}
\item
The total reflex pair $(E^\sharp,\Phi^\sharp)$ is CM pair.
\item
There are  natural homomorphisms 
\[
 \mathrm{Nm}^\sharp : E^\times \to ( E^\sharp)^\times ,\qquad 
 \mathrm{Tr}^\sharp : E\to E^\sharp,
\]
called the \emph{total reflex norm} and \emph{total reflex trace}, respectively.
\item
If we hold $E$ fixed but change $\iota_0$, the CM algebra $E^\sharp$ is unchanged, and $\Phi^\sharp$ is replaced by a CM type in the same $G_\Q$-orbit.
\end{enumerate}
\end{proposition}

\begin{proof}
The third claim is elementary, and left to the reader.
To prove the first two claims, we relate our total reflex pair to the classical notions of reflex field and reflex type.

Fix  $G_\Q$-orbit representatives
\[
\Phi_1,\ldots, \Phi_r \in \{\mbox{CM types of $E$} \},
\] 
and let $\mathrm{Stab}(\Phi_i) \subset G_\Q$ be the stabilizer of $\Phi_i$.
For each   CM  pair $(E,\Phi_i)$ let $(E_i^\prime , \Phi_i^\prime)$  be the classical reflex  CM pair, defined by 
\[
E_i^\prime = \{ x \in \Q^\alg : \sigma(x) =x ,\, \forall \sigma \in  \mathrm{Stab}(\Phi_i)\}.
\]
and 
\[
\Phi_i^\prime = \{ \sigma^{-1}|_{E_i^\prime} : \sigma \in G_\Q, \,   \sigma \circ \iota_0 \in \Phi_i \}  \subset \Hom( E_i^\prime ,\C ).
\]
It is an exercise in Galois theory to show that there is an isomorphism of $\Q$-algebras
\[
E^\sharp \iso  E_1^\prime  \times \cdots \times E_r' 
\] 
such that the natural bijection
\[
\Hom(E^\sharp , \C) \iso  \Hom( E_1^\prime ,\C)\sqcup \cdots \sqcup\Hom( E_r^\prime ,\C)
\]
identifies $\Phi^\sharp = \Phi_1 ^\prime\sqcup \cdots \sqcup \Phi_r^\prime$.

It is known from the classical theory that each $(E',\Phi')$ is a CM type, and hence so is $(E^\sharp,\Phi^\sharp)$.
Moreover,  the classical theory provides homomorphsms
\begin{align*}
\mathrm{Nm}_{\Phi_i} &= \prod_{\varphi\in \Phi_i} \varphi : E^\times \to ( E_i' )^\times \\
\mathrm{Tr}_{\Phi_i}  &= \sum_{\varphi\in \Phi_i} \varphi : E  \to  E_i' ,
\end{align*}
and the total reflex norm and total reflex trace are constructed from these in the obvious way using the isomorphism
$E^\sharp \iso  E_1^\prime  \times \cdots \times E_r' $.
\end{proof}


\subsection{CM abelian varieties}


Suppose $A$ is an abelian variety  defined over a number field $\kk$, and assume that $A$ extends to a semi-abelian scheme
\[
A \to \Spec(\co_\kk)
\]
 in the sense of \cite{BLR}.   For example, if $A$ has everywhere good reduction then its N\'eron model provides such an extension.

The Faltings height of $A$ is defined exactly as in the case of an elliptic curve.
Pick a nonzero rational section $\eta$ of the line bundle $\pi_*\Omega^{\dim(A)}_{A/\co_\kk }$ on $\Spec(\co_\kk)$.
Define the \emph{archimedean part of the Faltings height}
\[
h^\Falt _\infty ( A, \eta ) = \frac{-1}{ 2 [\kk:\Q] }  \sum_{ \sigma : \kk \to \C}
\log \big|   \int_{ A^\sigma(\C) } \eta^\sigma \wedge \overline{\eta^\sigma}\,  \big|,
\]
where $A^\sigma$ is the base change of $A$ via $\sigma : \kk \to \C$, and similarly for $\eta^\sigma$.
Define the \emph{finite part of the Faltings height}
\[
h^\Falt_f(A,\eta) = \frac{1}{  [\kk:\Q] }  \sum_{ \mathfrak{p} \subset \co_\kk}  \ord_\mathfrak{p}(\eta) \cdot  \log ( \mathrm{N}(\mathfrak{p})  ) ,
\]
where $\ord_\mathfrak{p}$ is defined by choosing an isomorphism of $\co_{\kk,\mathfrak{p}}$-modules
\[
H^0\big( \Spec(\co_\kk) , \pi_*\Omega^{\dim(A)}_{A/\co_\kk } \big) \otimes_{\co_\kk} \co_{\kk,\mathfrak{p}} \iso \co_{\kk,\mathfrak{p}}.
\]
The \emph{Faltings height}
\begin{equation}\label{faltings def}
h^\Falt(A) = h^\Falt_f(A, \eta) + h^\Falt _\infty ( A,\eta) 
\end{equation}
is independent of  the choice of  $\eta$, and is unchanged if we enlarge the number field $\kk$.  This allows us to define the Faltings height of any abelian variety over $\Q^\alg$, by choosing a model over a number field with everywhere semi-abelian reduction.

Now fix a CM algebra $E$ of dimension $2d$, and assume that $A$ is an abelian variety of dimension $d$ defined over  $\C$,  equipped with an injective ring homomorphism
\begin{equation}\label{CM action}
E \to \End(A) \otimes_\Z \Q.
\end{equation}
Such a homomorphism endows the homology group  $H_1(A(\C) , \Q)$ with the structure of a free $E$-module of rank one.  Using this and the Hodge decomposition
\[
H_1(A(\C) , \C) \iso \Lie(A) \oplus \overline{ \Lie(A) },
\]
it is easy to see that there is a unique CM type $\Phi$ such that 
\[
\Lie(A) \iso \prod_{\varphi \in \Phi} \C_\varphi
\]
as $E\otimes_\Q \C$ modules, where $\C_\varphi=\C$ with $E$ acting through $\varphi:E\to\C$.  

In the situation above, we say that $A$ has \emph{complex multiplication  of type $(E,\Phi)$}. 
If (\ref{CM action}) restricts to a homomorphism $\co_E \to \End(A)$,
we say that $A$ has \emph{complex multiplication of type $(\co_E,\Phi)$}.

The following proposition is standard.  See for example the statement and proof of  Proposition 1.1 in Chapter 5 of \cite{La83}.

\begin{proposition}
Suppose $A$ is an abelian variety over $\C$ with  complex multiplication of type $(E,\Phi)$.
\begin{enumerate}
\item There is a model of $A$ defined over a number field, and this model may be chosen to have everywhere good reduction.
\item Any other abelian variety over $\C$ with complex multiplication of type $(E,\Phi)$ is $E$-linearly isogenous to $A$.
\end{enumerate}
\end{proposition}

\begin{theorem}[Colmez \cite{Colmez}]\label{thm:colmez}
Suppose $A$ is an abelian variety over $\Q^\alg$ with  complex multiplication of type $(\co_E,\Phi)$. The Faltings height 
\[
 h^\Falt_{ (E,\Phi) } =  h^\Falt(A)
\]
 depends only on the pair $(E,\Phi)$, and not on $A$ itself.  It is unchanged if $\Phi$ is replaced by a CM type in the same $G_\Q$-orbit.\end{theorem}


\subsection{Colmez's conjecture}


Fix a CM pair  $(E,\Phi)$.  
Theorem \ref{thm:colmez} says that the Faltings height $ h^\Falt_{ (E,\Phi) }$ depends only on the Galois-theoretic data $(E,\Phi)$.
It is natural to ask if there if there is some formula for it that makes this more transparent.  This is precisely what Colmez's conjecture provides: an expression for  $h^\Falt_{ (E,\Phi) }$ in terms of Artin $L$-functions, that makes no mention of abelian varieties.

Define a complex-valued function on $G_\Q$ by 
\[
A_{( E,\Phi) } (g)  = | \Phi \cap g \circ  \Phi  |.
\]  
By averaging over the $G_\Q$-orbit of $\Phi$, we obtain a function 
\[
 A^0_{( E,\Phi)} = \frac{1}{[  G_\Q : \mathrm{Stab}(\Phi) ] } \sum_{ g \in G_\Q/\mathrm{Stab}(\Phi) } A_{(E,g \circ \Phi )}
\]
on $G_\Q$ that is  locally constant, and constant on conjugacy classes. 
As such there is a unique decomposition
\[
A^0_{(E,\Phi)} = \sum_\chi m_{(E,\Phi)}(\chi) \cdot \chi
\]
 as a linear combination of Artin characters.  That is, each $\chi : G_\Q \to \C$ is the character of a continuous representation 
$\rho_\chi : G_\Q \to \GL(W_\chi)$ on a finite dimensional complex vector space.    For each such Artin character  $\chi$ let 
\[
L(s,\chi) = \prod_p   \frac{1}{   \det\big( 1 - p^{-s} \cdot    \rho_\chi( \mathrm{Fr}_p )   \big)   } 
\] 
be the usual Artin $L$-function, where the product is over all primes $p<\infty$ of $\Q$,   and 
\[
 \rho_\chi( \mathrm{Fr}_p ) :  W^{I_p  }  \to  W^{I_p  } 
\]
is  Frobenius acting on the subspace  $W^{I_p} \subset W$  of vectors fixed by the inertia subgroup $I_p \subset G_\Q$
of some chosen place of $\Q^\alg$ above $p$.

Let $c\in G_\Q$ be complex conjugation.  Using the observation that
\[
A^0_{(E,\Phi)}(g) + A^0_{(E,\Phi)}( c\circ g ) = |\Phi |
\]
is independent of $g$, one can  check that any  nontrivial Artin character with  $m_{(E,\Phi)}(\chi) \not=0$  must
satisfy $\chi(c) = -\chi(\mathrm{id})$.  Hence, by   Proposition 3.4 in  Chapitre I of  \cite{Tate}, the $L$-function $L(s,\chi)$ has neither a zero nor a pole at $s=0$.   Define
\[
Z_{(E,\Phi)} = \sum_\chi m_{(E,\Phi)}(\chi) \frac{L'(0,\chi)}{L(0,\chi)} 
\]
and
\[
\mu_{(E,\Phi)} = \sum_\chi m_{(E,\Phi)}(\chi)  \log( f_\chi) ,
\]
where the sums are over all Artin characters, and $f_\chi$ is the Artin conductor of $\chi$.  
A good reference for Artin representations, including the definition of the Artin conductor, is \cite{Murty}.
The \emph{Colmez height} of  the CM pair  $(E,\Phi)$ is defined by
\[
h^\Colmez_{ (E,\Phi) }= - Z_{(E,\Phi)} -  \frac{ \mu_{(E,\Phi)}}{2}.
\]

The following is our restatement of Conjecture 0.4 of  \cite{Colmez}.

\begin{conjecture}[Colmez]\label{conj:general colmez}
For any CM pair $(E,\Phi)$ we have
\[
 h^\Falt_{ (E,\Phi) } = h^\Colmez_{ (E,\Phi) } .
\]
\end{conjecture}

Following ideas of Gross \cite{gross:unitary} and Anderson \cite{And82}, Colmez  was able to prove the abelian case of Conjecture \ref{conj:general colmez} up to a rational multiple  of $\log(2)$.  The $\log(2)$ error term was later removed  by Obus.

\begin{theorem}[Colmez \cite{Colmez}, Obus \cite{obus}]
If $E/\Q$ is Galois with abelian Galois group, then
\[
 h^\Falt_{ (E,\Phi) } = h^\Colmez_{ (E,\Phi) } 
\]
holds for every CM type  $\Phi \subset \Hom( E,\C)$.
\end{theorem}



\subsection{The averaged version}


Let $E$ be a CM field of degree $2d$.  Let $F\subset E$ be its maximal totally real subfield, and let 
\[
\chi : \A_F^\times \to \{\pm 1\}
\]
be the associated quadratic character. 
Denote by $L(s,\chi)$ the usual $L$-function, and by 
\[
\Lambda(s, \chi )  =   \left|   D_E   /  D_F \right|^{\frac{s}{2}}   \cdot   \Gamma_\R(s+1)^d  \cdot L(s,\chi)
\]
the completed $L$-function.  Here   
\[
\Gamma_\R(s) =      \pi^{-s/2} \Gamma(s/2) ,
\]
and $D_E$ and $D_F$ are the (absolute) discriminants of $E$ and $F$.   
The completed $L$-function satisfies the function equation  
\[
\Lambda(1-s, \chi) = \Lambda(s,\chi),
\]
and
\[
\frac{ \Lambda'(0,\chi )  }{\Lambda(0,\chi )}    
=    \frac{ L'(0,\chi)  }{ L(0,\chi) }  + \frac{1}{2} \log \left| \frac{  D_{E}  }{  D_{F} }\right| 
- \frac{ d   \log(4\pi e^\gamma)}{ 2 } .
\]

As Colmez himself noted, the right hand side of the identity of Conjecture \ref{conj:general colmez} simplifies considerably if one averages over all CM types of $E$.  
The following is an elementary exercise.  See the proof of  \cite[Proposition 9.3.1]{AGHMP-2} for details.

\begin{proposition}\label{prop:artin average}
For a fixed CM field $E$ of degree $2d$,
\[
\frac{1}{2^d} \sum_\Phi h^\Colmez_{(E,\Phi)}
 = 
-  \frac{1}{2} \cdot \frac{ L'(0,\chi)  }{ L(0,\chi) }  -   \frac{1}{4}  \cdot  \log \left| \frac{  D_{E}  }{  D_{F} }\right| 
-  \frac{ d   \log(2\pi)}{ 2 }
\]
where the summation on the left is over all CM types $\Phi\subset \Hom(E,\C)$.
\end{proposition}

Proposition \ref{prop:artin average} is the justification for calling Theorem \ref{thm:average colmez} the averaged Colmez conjecture.
Somewhat oddly, nowhere in our proof of Theorem \ref{thm:average colmez} does there ever appear an abelian variety with complex multiplication by $E$.
What will appear are  abelian varieties with complex multiplication by the total reflex algebra $E^\sharp$ defined in \S \ref{ss:reflex}.
Recall that the total reflex type $\Phi^\sharp$ depends on the additional choice of an embedding $\iota_0 : E \to \C$, but the $G_\Q$-orbit of $\Phi^\sharp$ does not. It follows that  $A^0_{( E^\sharp ,\Phi^\sharp) }$,  $h^\Falt_{( E^\sharp ,\Phi^\sharp) }$, and  $h^\Colmez_{( E^\sharp ,\Phi^\sharp) }$ depend only on $E$, and not on $\iota_0$.

The connection between abelian varieties with CM by $E$ and CM by $E^\sharp$ is proved by the following  proposition.
The proof is elementary Galois theory, and we refer the reader to  \cite[Proposition 9.3.2]{AGHMP-2} for the proof.

\begin{proposition}\label{prop:average reflex}
For any CM field $E$ we have 
\[
 A^0_{( E^\sharp ,\Phi^\sharp) } = \frac{ 1}{ 2d }   \sum_{ \Phi } A^0_{( E,\Phi) }  ,
\]
where the summation on the right is over all CM types of $E$. 
\end{proposition}

\begin{corollary}\label{cor:reflex height}
For any CM field $E$ we have 
\[
 h^\Falt_{( E^\sharp ,\Phi^\sharp) }  =\frac{ 1}{ 2d }   \sum_\Phi h^\Falt_{(E,\Phi)} ,\qquad
  h^\Colmez_{( E^\sharp ,\Phi^\sharp) }  =\frac{ 1}{ 2d }   \sum_\Phi h^\Colmez_{(E,\Phi)}
\]
where the summations are over all CM types of $E$. 
\end{corollary}

\begin{proof}
The Colmez height of  $(E,\Phi)$ depends only on the function $A^0_{(E,\Phi)}$, and the dependence is linear.
The same is true of the Faltings height, but this is a nontrivial theorem of Colmez   \cite[Th\'eor\`eme 0.3]{Colmez}.
   Both equalities therefore  follow from Proposition \ref{prop:average reflex}.
\end{proof}


\section{Orthogonal Shimura varieties}
\label{s:orthogonal shimura variety}


Let $(V, Q)$ be a quadratic space  over $\Q$ of signature $(n,2)$, with $n\ge 1$.   The associated bilinear form is denoted
\begin{equation}\label{bilinear}
[x_1,x_2] = Q(x_1+x_2) - Q(x_1) - Q(x_2) .
\end{equation}
Fix a maximal lattice $L\subset V$.

In this section we explain how to attach to this data a Shimura variety $M$ over $\Q$.
This Shimura variety is not of PEL type, so does not have any simple interpretation as a moduli space of polarized abelian varieties.
It is, however, of Hodge type, which means that it can be embedded into a Siegel moduli space parametrizing polarized abelian varieties.

By pulling back the universal family we obtain the \emph{Kuga-Satake abelian scheme}  $A\to M$.
The Shimura variety carries over it a  \emph{line bundle of weight one modular forms}  $\taut$, 
 endowed with a Petersson metric.   We will see that this line bundle can be used to compute the Faltings heights of  the fibers of $A\to M$.


\subsection{Spinor similitudes}
\label{ss:gspin}


We recall some basic definitions in the theory of Clifford algebras and spinor groups.  More details may be found in  \cite{Shimura}.

The Clifford algebra $C=C(V)$  is  the quotient   of the full tensor algebra of $V$ by the two-sided ideal generated by all elements  $x\otimes x -Q(x)$ with $x\in V$. 
It  carries a   $\Z/2\Z$-grading 
\[
C=C^+ \oplus C^-
\] 
inherited from the usual  grading on the tensor algebra.
 The natural $\Q$-linear map $V\to C$ is injective, and we use it to regard
$
V\subset  C^-
$
as a subspace.
The Clifford algebra has dimension $2^{\mathrm{dim}(V)}$ as a $\Q$-vector space, and is   generated as a $\Q$-algebra by $V\subset  C^-$.

The quadratic space $(V,Q)$ has an associated \emph{spinor similitude group}\footnote{In the terminology of \cite{Shimura}, this is the \emph{even Clifford group}.} 
\[
G=\GSpin(V).
\]
It is an algebraic group over $\Q$ with rational points 
\[
G(\Q) = \{ g\in ( C^+ ) ^\times :  gV g^{-1} = V \}.
\]
The group $G$ acts on $V$ via the \emph{standard representation} $g\action v = gvg^{-1}$, and this action  determines a short exact sequence of algebraic groups
\[
 1 \to \mathbb{G}_m \to G \map{g\mapsto g\action} \SO(V) \to 1.
\]

 There is a unique $\Q$-linear involution $c\mapsto c^*$ on $C$   characterized by 
  \[  ( v_1\cdots v_r )^* = v_r^* \cdots v_1^* \]    for $v_1,\ldots, v_r\in V$.
One can show that every $g\in G(\Q)$ satisfies $g^* g \in \Q^\times$, and   the \emph{spinor similitude}
\[
\nu: G \to \mathbb{G}_m
\]
is defined by  $\nu(g) = g^* g$.  Its kernel  is the usual spin double cover of $\SO(V)$, and 
its  restriction  to the central $\mathbb{G}_m$ sends $z\mapsto z^2$.


\subsection{The orthogonal Shimura variety}
\label{ss:the shimura variety}


We construct a Shimura variety from the data   $(V,Q)$ and  the choice of maximal lattice $L\subset V$.

Using the standard representation $G \to \SO(V)$,  the group of real points $G(\R)$ acts on the $n$-dimensional hermitian symmetric domain
\begin{equation}\label{domain}
\mathcal{D} = \{ z\in V_\C : [z,z] =0 ,\, [z,\overline{z} ] <0 \} / \C^\times \subset \mathbb{P}(V_\C).
\end{equation}
There are two connected components $\mathcal{D}=\mathcal{D}^+ \sqcup \mathcal{D}^-$, interchanged by the action of 
any $\gamma\in G(\R)$ with $\nu(\gamma)<0$.

The pair $(G,\mathcal{D})$ is a Shimura datum with reflex field $\Q$.  More precisely, we can realize 
 \[
 \mathcal{D} \subset \Hom(\mathbb{S} , G_\R ) 
 \]
 as a $G(\R)$-conjugacy class as follows.  Given a point $z\in \mathcal{D}$, write $z=x+i y$ for vectors $x,y\in V_\R$.  
 The span $\R x+\R y \subset V_\R$ is a negative definite plane, with its own Clifford algebra $C_z \subset C_\R$.
 If we identify $\C \iso C_z^+$ as $\R$-algebras using 
 \[
 i \mapsto \frac{xy}{ \sqrt{Q(x)Q(y)}},
 \]
 we obtain a homomorphism $\C \to C_\R^+$, which restricts to a  homomorphism $h_z :  \C^\times  \to G(\R)$.
 This defines an element 
 $
 h_z  \in \Hom(  \mathbb{S}  , G_\R  ).
 $

The choice of maximal lattice $L\subset V$ determines  a  compact open subgroup of $G(\A_f)$ as follows.
Let 
\begin{equation}\label{clifford level}
C_{\widehat{\Z}} \subset C_{\A_f}
\end{equation}
 be the $\widehat{\Z}$-subalgebra generated by  the profinite completion  $\widehat{L} \subset V_{\A_f}$, and form   the intersection
 \[
 K = C_{\widehat{\Z}}^\times \cap G(\A_f)
 \]
inside of $C_{\A_f}^\times$.  This compact open subgroup of $G(\A_f)$ stabilizes $\widehat{L}$, and acts trivially on the discriminant group
\begin{equation}\label{disc group}
 \widehat{L}^\vee / \widehat{L} \iso L^\vee / L . 
\end{equation}
Here $L^\vee$ is the dual lattice of $L$ relative to the bilinear form (\ref{bilinear}).

By Deligne's theory of canonical models of Shimura varieties, the  orbifold
\begin{equation*}
M(\C) = G(\Q) \backslash \mathcal{D} \times G(\A_f) / K
\end{equation*}
is the space of complex points of a  Deligne-Mumford stack $M$ over $\Q$.  
Using the complex uniformization, one can check that $M$ has dimension $n= \dim(V)-2$.


\subsection{The Kuga-Satake abelian scheme}
\label{ss:symplectic}


There is a   $\Q$-linear  \emph{reduced trace} 
\[
\mathrm{Trd} : C \to \Q .
\]
If  $\mathrm{dim}(V)=2m$ is even,  the Clifford algebra is $C$ is a central simple $\Q$-algebra, and the 
 $\Q^\alg$-linear extension 
\[
M_{2^m}(\Q^\alg) \iso C \otimes_\Q \Q^\alg \map{\mathrm{Trd}} \Q^\alg
\]
of the reduced trace is the usual trace on matrices.
If $\mathrm{dim}(V)=2m+1$ is odd, the center of $C$ is a quadratic \'etale $\Q$-algebra (so either $\Q\oplus \Q$ or a quadratic field extension), and $C$ is a central simple algebra over its center.  
The $\Q^\alg$-linear extension 
\[
M_{2^m}(\Q^\alg) \oplus M_{2^m}(\Q^\alg)  \iso C \otimes_\Q \Q^\alg \map{\mathrm{Trd}} \Q^\alg
\]
of the reduced trace is the sum of the usual traces on matrices.

Let $H=C$, viewed as a vector space over $\Q$.  Choose vectors 
$e,f \in V$ such that $[e,f] =0$, and 
\[
Q(e) <0 ,\quad Q(f) < 0.
\]
The element
$
\delta = e f \in C^\times
$
satisfies  $\delta^* = -\delta$, which allows us to define a  symplectic form
\[
\psi_\delta : H\otimes H \to \Q
\]
by $\psi_\delta(h_1,h_2) = \mathrm{Trd}( h_1 \delta h_2^*)$.     
We may rescale $e$ and $f$ in order to assume that $\psi_\delta$ is $\widehat{\Z}$-valued on  the $\widehat{\Z}$-submodule
\begin{equation}\label{KS integral}
H_{\widehat{\Z}} \subset H_{\A_f} 
\end{equation}
determined by  (\ref{clifford level}).

Letting $G(\Q) \subset C^\times$ act on $H$ by left multiplication, we obtain a closed immersion of algebraic groups
\[
G \to \GSp(H),
\]
under which the symplectic similitude on $\GSp(H)$ restricts to the spinor similitude on $G$, and the action of $K\subset G(\A_f)$ stabilizes 
the $\widehat{\Z}$-lattice (\ref{KS integral}).   
This closed immersion determines a morphism from $(G,\mathcal{D})$ to the Siegel Shimura datum determined by $(H,\psi_\delta)$, and hence determines  a morphism 
\[
M \to X
\]
to the $\Q$-stack  $X$ parametrizing abelian varieties of dimension 
\[
2^{n+1}=  \frac{ \dim(H) }{2}
\]
 equipped with a polarization of some fixed degree (depending on $\psi_\delta$).  Pulling back the universal object over $X$, we obtain the \emph{Kuga-Satake abelian scheme}
\[
A \to M.
\]
It carries a polarization that depends on the choice of $\psi_\delta$, but the underlying abelian scheme is independent of the choice.


\subsection{Automorphic vector bundles}
\label{ss:bundles}


Like any Shimura variety, $M$  has a theory of automorphic vector bundles \cite{harris:automorphic_0,harris:automorphic_1,harris:automorphic_2,milne}.

For us this means the following: any representation $G \to \GL(N)$  on a finite dimensional $\Q$-vector space has a \emph{de Rham} realization, which is vector bundle $\bm{N}_\dR$ on $M$ endowed with a decreasing filtration $F^\bullet \bm{N}_\dR$ by local direct summands.
At a complex point $[ z,g] \in M(\C)$ the fiber of $\bm{N}_\dR$ is just the vector space $N_\C$ endowed with the Hodge  filtration determined by  
\[
\mathbb{S}  \map{h_z} G(\R) \to \GL(N_\R).
\]

The representation $G\to \GL(N)$ also has a \emph{Betti realization}
\[
\bm{N}_B  = G(\Q) \backslash  \mathcal{D}  \times N  \times G(\A_f)  / K.
\]
This is a local system of $\Q$-vector spaces on the complex fiber $M(\C)$, which is related to the de Rham realization by a canonical isomorphism
\[
\bm{N}_\dR(\C)  \iso  \bm{N}_B  \otimes_\Q \co_{M(\C)}
\]
of  holomorphic vector bundles on the complex fiber.  
A choice of  $K$-stable $\widehat{\Z}$-lattice
$
N_{\widehat{\Z}} \subset N_{\A_f}
$
determines a local subsystem of $\Z$-modules \[\bm{N}_{B,\Z} \subset \bm{N}_B,\]
whose fiber at a complex point $[z,g] \in M(\C)$ is the $\Z$-lattice $ g  N_{\widehat{\Z}} \cap N$ in the $\Q$ vector space  $N$.
Of  course the intersection here takes place in $N_{\A_f}$.

We  apply these constructions  to the  representations 
\begin{equation}\label{two reps}
G \to \SO(V) ,\qquad G \to \GSp(H)
\end{equation}
to obtain vector bundles $\bm{V}_\dR$ and $\bm{H}_\dR$ on $M$, along with filtrations
\[
0 = F^2 \bm{V}_\dR \subsetneq F^1 \bm{V}_\dR \subsetneq F^0 \bm{V}_\dR \subsetneq F^{-1} \bm{V}_\dR =\bm{V}_\dR
\]
and
\[
0 = F^1 \bm{H}_\dR \subsetneq F^0 \bm{H}_\dR \subsetneq F^{-1}  \bm{H}_\dR =\bm{H}_\dR.
\]
The symmetric bilinear pairing on $V$  induces a symmetric bilinear pairing
\[
[\cdot,\cdot] : \bm{V}_\dR \otimes \bm{V}_\dR \to \co_{M}
\]
under which $F^1 \bm{V}_\dR$ is an isotropic line, and $F^0\bm{V}_{dR} = ( F^1\bm{V}_\dR)^\perp$.
Similarly,  the symplectic form on $H$ induces an alternating  bilinear pairing
\[
\psi_\delta : \bm{H}_\dR \otimes \bm{H}_\dR \to \co_{M}
\]
under which $F^0 \bm{H}_\dR$ is maximal isotropic.
Among these vector bundles, the line bundle $F^1 \bm{V}_\dR$ will play a distinguished role.

\begin{definition}
The \emph{line bundle of weight one modular forms}  is
\[\taut = F^1 \bm{V}_\dR.\]
\end{definition}

The vector bundles $\bm{V}_\dR$ and $\bm{H}_\dR$ are closely related.
In fact, we can identify 
\begin{equation}\label{rep inclusion}
V\subset \End(H)
\end{equation}
as a $G$-stable subspace using  the left  multiplication action of $V\subset C$ on $C=H$.   
There is a corresponding  inclusion 
\[
\bm{V}_\dR \subset  \underline{\End}(\bm{H}_\dR)
\]
as a local direct summand,  respecting the natural filtrations.
  If $x_1$ and $x_2$ are local sections of $\bm{V}_\dR$, viewed as endomorphisms of $\bm{H}_\dR$, then
\[
[x_1,x_2] = x_1\circ x_2 + x_2\circ x_1
\]
as local sections of $\co_{M} \subset  \underline{\End}(\bm{H}_\dR)$.

At a complex point $s=[z,g] \in M(\C)$, all of this can be made completely explicit in terms of the isotropic vector $z\in V_\C$.
The filtration on $\bm{V}_{\dR,s} = V_\C$ is given by
\[
F^2 \bm{V}_{dR,s} =0, \quad F^1 \bm{V}_{dR,s} =\C z ,\quad F^0 \bm{V}_{dR,s} = (\C z)^\perp ,\quad F^{-1}\bm{V}_{dR,s} = V_\C,
\]
while the filtration on $\bm{H}_{\dR,s} = H_\C$ is given by
\[
F^1 \bm{H}_{dR,s} =0, \quad F^0 \bm{H}_{dR,s} = z H_\C ,\quad  F^{-1}\bm{H}_{dR,s} = H_\C.
\]

In particular, suppose we pull back the line bundle of weight one modular forms $\taut = F^1\bm{H}_\dR$ via the complex uniformization
\[
\mathcal{D} \map{ z\mapsto [z,g] } M(\C)
\]
for some fixed $g\in G(\A_f)$.  The pullback is just the restriction to $\mathcal{D} \subset \mathbb{P}(V_\C)$ of the tautological bundle on projective space.
In other words,  the fiber of the line bundle $\taut$ at a complex point
$[z,g] \in M(\C)$  is the isotropic  line 
\[
\taut_{ [z,g] } = \C z \subset V_\C.
\]
   We can use this identification to define the 
\emph{Petersson metric} on $\taut_{ [ z,g]}$  by
\begin{equation}\label{pet metric}
|| z ||^2 =  - [z,\overline{z}].
\end{equation}


\subsection{Special divisors}
\label{ss:divisors}


We now define a family of divisors on $M$, following  Kudla \cite{Ku97,Ku04}.  
These divisors also play a prominent role in work of  Borcherds  \cite{Bor98} and Bruinier \cite{Bruinier}.

For $x\in V$ of positive length, define an analytic divisor on  (\ref{domain}) by
\[
\mathcal{D}_x = \{ z \in \mathcal{D} : z \perp x \}.
\]
For every  positive  $m\in \Q$ and every $\mu\in L^\vee/L$ we define a complex orbifold
\begin{equation}\label{eqn:Zmmu}
Z(m, \mu ) (\C)= \bigsqcup_{ g\in G(\Q)\backslash G(\A_f) /K }  \Gamma_g \big\backslash \Big(
\bigsqcup_{  \substack{  x\in  \mu_g+ L_g \\ Q(x)=m  }    }   \mathcal{D}_x
\Big).
\end{equation}
Here we have set  
\[
\Gamma_g = G(\Q) \cap g K g^{-1}.
\]
Recalling that the action $G \to \SO(V)$ is denoted $g\mapsto g\action$,  the $\Z$-lattice    $L_g \subset V$ is defined  by 
  \begin{equation}\label{shifted lattice}
  \widehat{L}_g = g\action \widehat{L},
  \end{equation}
    and we have set  
    \[
    \mu_g=g\action \mu\in L_g^\vee/L_g.
    \]

    We now use the Kuga-Satake abelian scheme $A\to M$ of \S \ref{ss:symplectic}   to give a more moduli-theoretic  interpretation of (\ref{eqn:Zmmu}).
As explained in \S \ref{ss:bundles}, the $\widehat{\Z}$-lattice (\ref{KS integral})  determines a local system of $\Z$-modules
$\bm{H}_{B,\Z}$ on $M(\C)$, along with an isomorphism
\[
\bm{H}_{B,\Z} \otimes_\Z \co_{M(\C)}  \iso \bm{H}_\dR(\C)
\]
of holomorphic vector bundles on $M(\C)$.  The Kuga-Satake abelian scheme over $M(\C)$ can be identified with the  analytic family of complex tori
\begin{equation}\label{analytic KS}
A(\C) = \bm{H}_{B,\Z} \backslash \bm{H}_\dR(\C) /F^0 \bm{H}_\dR(\C) .
\end{equation}

Similarly,  the  $K$-stable $\widehat{\Z}$-lattices $\widehat{L} \subset \widehat{L}^\vee$ in   $V_{\A_f}$  determine local systems of $\Z$-modules
\[
\bm{L}_B \subset \bm{L}_B^\vee
\]
inside  the local system of $\Q$-modules $\bm{V}_B$ determined by $G \to \SO(V)$.  Using the fact that $K$ acts trivially on the discriminant group (\ref{disc group}), one obtains a canonical isomorphism of local systems of abelian groups
\begin{equation}\label{coset trivialization}
\bm{L}_B^\vee / \bm{L}_B \iso (L^\vee / L) \otimes_\Z \underline{\Z}.
\end{equation}

Suppose we have a complex point $s\in M(\C)$, and an quasi-endomorphism
\[
x\in \End(A_s)_\Q
\]
of the fiber of the Kuga-Satake abelian scheme.  Using (\ref{analytic KS}), one obtains an induced endomorphism 
$x_B \in \End_\Q(  \bm{H}_{B,s} )$,  called the \emph{Betti realization} of $x$. 
We say that  $x$ is \emph{special}  if its  Betti realization lies in the $\Q$-subspace
\[
\bm{V}_{ B,s} \subset \End_\Q(  \bm{H}_{B,s} ).
\]
The space
\[
V(A_s) = \{ x\in \End(A_s) : x \mbox{ is special} \}
\]
of all special endomorphisms of $A_s$ is a free $\Z$-module of finite rank, endowed with a positive definite $\Z$-valued quadratic form characterized by the equality $Q(x) = x\circ x$ in $\End(A_s)$.

\begin{remark}
The rank of $V(A_s)$ behaves erratically as $s \in M(\C)$ varies.  It can be a small as $0$, or as large as $\mathrm{dim}(V)-1=n+1$.
\end{remark}

More generally,  for any $\mu \in L^\vee / L$, denote by 
\[
V_\mu(A_s) \subset \End(A_s)_\Q
\]
the set of  special quasi-endomorphisms $x$ whose Betti realizations
$
x_B \in \bm{V}_{ B,s}
$
lie in the $\Z$-submodule $\bm{L}^\vee_{B, s}$, and which are sent to $\mu$  under the isomorphism 
\[
\bm{L}^\vee_{B, s}  /   \bm{L}_{B, s}  \iso  L^\vee / L
\]
induced by  (\ref{coset trivialization}).  In particular
\[
V_\mu(A_s) \subset V(A_s)_\Q.
\]
Taking  $\mu=0$ recovers $V(A_s)$.

The proof of the following proposition is a straightforward exercise.

\begin{proposition}\label{prop:complex divisor}
For any positive rational number $m$ and any $\mu \in L^\vee / L$, the set of (isomorphism classes of) complex points $Z(m,\mu)(\C)$ is in bijection with the set of pairs $(s,x)$ consisting of a point $s\in M(\C)$, and a special quasi-endomorphism  $x\in V_\mu(A_s)$ such that  $Q(x) = m$.
\end{proposition}

The orbifold (\ref{eqn:Zmmu}) can be realized as the space of complex points of a smooth Deligne-Mumford stack over $\Q$ of dimension $n-1$,  endowed with a finite and unramified morphism
\[
Z(m, \mu )  \to M  .
\]
This is a special case of the more general constructions of \S \ref{ss:integral divisors}.

Even though this morphism is not a closed immersion, one can construct from it a Cartier divisor on $M$ as follows.
Every geometric point of $M$ admits an \'etale neighborhood $U\to M$ small enough that the pullback $Z(m,\mu)_U \to M_U$ restricts to a closed immersion on every connected component of its domain, and each such closed immersion is defined locally by a single nonzero equation.  Thus every connected component of $Z(m,\mu)_U$ determines an effective Cartier divisor on $M_U$.  Summing over all connected components,  and then glueing over an \'etale cover,  defines an effective Cartier divisor  on $M$.
We will refer to both the stacks $Z(m,\mu)$ and their associated Cartier divisors as \emph{special divisors} on $M$.


\subsection{The connection to Faltings heights}
\label{ss:ks heights}


We want to explain how the line bundle $\taut$, endowed with its Petersson metric (\ref{pet metric}), can be used to compute Faltings heights of fibers of the Kuga-Satake abelian scheme 
\[
A\to M,
\]
in much the same way that the line bundle of modular forms on the usual modular curve  computes Faltings heights of elliptic curves (\ref{elliptic faltings}).

There is a short exact sequence
\begin{equation}\label{KS hodge}
0 \to F^0 \bm{H}_\dR \to  \bm{H}_\dR \to \Lie(A) \to 0
\end{equation}
of vector bundles on $M$, whose fiber at every complex point $y\in M(\C)$ is canonically identified with the $\C$-linear dual of the Hodge short exact sequence
\[
0 \to H^0(  A_y,  \Omega_{A_y/\C} ) \to  H^1_\dR(A_y / \C)   \to H^1(A_y , \co_{A_y} )  \to 0.
\]
We will use this to  relate $\taut$ to the determinant of the dual Lie algebra 
\[
\Lie(A)^\vee \iso \pi_* \Omega_{ A/M }.
\]

\begin{proposition}\label{prop:bundles}
There is an isomorphism of line bundles
\[
\taut^{ \otimes \dim(A) }  \iso   \det( \pi_*\Omega_{A/M} )^{\otimes 2}   \otimes   \det (\bm{H}_\dR) .
\]
\end{proposition}

\begin{proof}
Recalling that 
\[
\taut =F^1 \bm{V}_\dR \subset F^1 \underline{\End}( \bm{H}_\dR ) \iso \underline{\Hom}( \bm{H}_\dR/ F^0 \bm{H}_\dR, F^0\bm{H}_\dR),
\]
there is a short exact sequence
\[
0 \to \taut \otimes F^0 \bm{H}_\dR \to \taut \otimes \bm{H}_\dR \map{ x \otimes a \mapsto xa } F^0 \bm{H}_\dR \to 0
\]
of vector bundles on $M$.  Taking determinants yields an isomorphism
\begin{equation}\label{some dets}
\taut^{ \otimes \dim(A) } \otimes   \det (\bm{H}_\dR )   \iso  \det( F^0 \bm{H}_\dR )^{ \otimes 2} .
\end{equation}
Explicitly,  if we set $r=2^{n+1}$, let  $a_1,\ldots, a_r$ be a local basis of  $F^0 \bm{H}_\dR$, and  extend it to a local basis $a_1,\ldots, a_r,b_1,\ldots, b_r$ of $\bm{H}_\dR$, then
\[
x^{\otimes \dim(A) } \otimes ( a_1\wedge\cdots\wedge a_r \wedge b_1\wedge\cdots \wedge b_r)
\mapsto (a_1\wedge\cdots\wedge a_r) \otimes (x b_1 \wedge \cdots \wedge x b_r)
\]
for any local generator $x$ of $\taut \subset \bm{V}_\dR \subset \underline{\End}( \bm{H}_\dR )$.

On the other hand,  taking determinants in (\ref{KS hodge}) yields an isomorphism
\[
\det(\bm{H}_\dR) \iso \det( F^0 \bm{H}_\dR) \otimes \det(\Lie(A)),
\]
which we rewrite as
\[
\det( F^0 \bm{H}_\dR)  \iso \det(\bm{H}_\dR)  \otimes  \det( \pi_*\Omega_{A/M} ).
\]
Substituting this expression into the right hand side of  (\ref{some dets}) gives the desired isomorphism.
\end{proof}

As in the introduction, $a\circeq b$ means that $a-b$ is a $\Q$-linear combination of $\{\log(p) : p \mbox{ prime}\}$.

\begin{theorem}\label{thm:modular height}
Suppose $\kk$ is a number field, and $y\in M(\kk)$.  If $v \in \taut_y$ is any nonzero vector, then
\[
 \frac{-1}{ [ \kk:\Q] } \sum_{ \sigma : \kk \to \C } \log || v^\sigma ||
  \circeq       2^{-n}  h^\Falt ( A_y) +   \log(2\pi)  .
\]
\end{theorem}

\begin{proof}
For any line bundle $\mathcal{L}$ on $M$ endowed with a metric $|| \cdot||$ on the complex fiber $M(\C)$, and any complex point $y\in M(\kk)$, abbreviate
\[
\log || \mathcal{L}_y || = 
 \frac{1}{ [ \kk:\Q] } \sum_{ \sigma : \kk \to \C } \log || v^\sigma ||
\]
for any nonzero vector $v\in \mathcal{L}_y$.  Of course this depends on the choice of $v$, but different choices 
change the value by a $\Q$-linear combination of $\log(p)$, yielding a well-defined element 
\[
\log || \mathcal{L}_y ||  \in \R / \circeq.
\]

For a complex point $y\in M(\C)$,  the \emph{Faltings metric} on the fiber at $y$ of $\det( \pi_* \Omega_{A/M} )$ is defined by
\[
|| \eta ||^2 = \big| \int_{ A_y(\C) } \eta \wedge \overline{\eta} \,  \big|
\]
for any top degree global holomorphic form $\eta \in    \det( \pi_*\Omega_{A_y/\C} )$ on $A_y(\C)$.
This makes $\det( \pi_* \Omega_{A/M} )$ into a metrized line bundle.

The local system
\[
\bm{H}_{B,\Z}  \subset \bm{H}_{B,\Z} \otimes_\Z \co_{M(\C)}  \iso \bm{H}_{\dR, M(\C)}
\]
of $\Z$-modules determines a rank $1$ local system 
\[
\det(\bm{H}_{B,\Z}) \subset \det( \bm{H}_{\dR, M(\C)} ),
\]
and   we define the \emph{volume metric} on $\det( \bm{H}_\dR )$ by  declaring that $||e||^2 = 1$ 
for any local generator  $e$ of $\det(\bm{H}_{B,\Z})$.  At a complex point $y\in M(\C)$ the dual volume metric on
\[
\det(\bm{H}_{\dR,y} ^\vee) \iso  \det(  H^1_\dR (A_y/\C)  )  \iso H^{2 \mathrm{dim}(A_y)}_\dR (A_y / \C)
\]
is just integration of top degree forms:
\[
||  u  ||^2 = \big| \int_{A_y(\C) } u \, \big|^2 .
\]

\begin{lemma}
If $\kk$ is a number field and $y\in M(\kk)$, then
\[
\log|| \det(\bm{H}_\dR)_y ||   \circeq   - \mathrm{dim}(A) \cdot \log(2\pi) 
\]
and 
\[
\log||   \det( \pi_* \Omega_{A/M} )_y||    \circeq  -h^\Falt(A_y).
\]
\end{lemma}

\begin{proof}
For the first claim,  suppose  we have a nonzero vector
\[
u \in     \det(\bm{H}_{dR,y} ^\vee)   \iso  \det ( H^1_\dR (A_y / \kk ) ) .
\]
As in \cite[page 22]{DMOS}, there is a  scalar $\mathrm{Tr}_\dR( u ) \in \kk$ such that 
\[
\frac{1}{ (2\pi i )^{\mathrm{dim}(A) } }   \int_{A_{y^\sigma}(\C) } u^\sigma = \sigma(\mathrm{Tr}_{dR} ( u )),
\]
for every $\sigma: \kk\to \C$.  This implies that 
\begin{align*}
\log||  \det(\bm{H}_\dR ^\vee) _y ||  
& =  \frac{1}{ [\kk:\Q] } \sum_{\sigma : \kk\to \C}  \log || u^\sigma ||  \\
  & =  \mathrm{dim}(A) \cdot  \log (2\pi) - \frac{1}{ [\kk:\Q]} \log |\mathrm{Nm}_{\kk/\Q}(\mathrm{Tr}_\dR( u ))|  \\
  & \circeq  \mathrm{dim}(A) \cdot \log(2\pi), 
\end{align*}
and the first claim   follows.

For the second claim, directly from the definition (\ref{faltings def}), we compute
\begin{align*}
 \log ||   \det( \pi_* \Omega_{A/M} )_y ||
&= 
 \frac{1}{2 [\kk:\Q] }  \sum_{ \sigma : \kk \to \C}
\log \big|   \int_{ A^\sigma(\C) } \eta^\sigma \wedge \overline{\eta^\sigma}\,  \big|\\
& =  - h^\Falt _\infty ( A, \eta )
\end{align*}
for  $\eta \in \det( \pi_*\Omega_{A_y/\kk} ) $  any nonzero vector.
Now use  $h^\Falt _f( A, \eta ) \circeq 0$.
\end{proof}

We can now complete the proof of  Theorem  \ref{thm:modular height}.
It is an exercise in linear algebra to show that the isomorphism of Proposition \ref{prop:bundles} respects the metrics,  and hence
\[
\dim(A) \cdot \log|| \taut_y || 
\circeq 2 \cdot \log ||   \det( \pi_*\Omega_{A/M} )_y||   +  \log||  \det (\bm{H}_\dR)_y||
\]

Combining this with the  lemmas  shows that
\[
\dim(A) \cdot \log|| \taut_y||   \circeq  - 2  \cdot  h^\Falt(A_y)  -  \dim(A) \cdot \log(2\pi)  .
\]
To complete the proof of the theorem, recall that $\dim(A)=2^{n+1}$.
\end{proof}


\section{Integral models}
\label{s:integral}


As in  \S \ref{s:orthogonal shimura variety},  let  $(V,Q)$ be a quadratic space over $\Q$ of signature $(n,2)$ with $n\ge 1$, and let $L\subset V$ be  a maximal lattice.  Let $M$ be the associated orthogonal Shimura variety of \S \ref{ss:the shimura variety}.


\subsection{Integral models in the self-dual case}
\label{ss:smooth integral}


Assume that $L^\vee =L$.

Recall that the  constructions of \S \ref{ss:symplectic}  provide us with a morphism 
\[
M\to X,
\]
where $X$ is the Deligne-Mumford stack over $\Q$ parametrizing abelian schemes of dimension $2^{n+1}$ endowed with a polarization of some fixed degree (depending on the choice of $\psi_\delta$).   
By virtue of its definition as a moduli space, it has a  natural extension to a stack  $\mathcal{X}$ over $\Z$, and 
we define $\mathcal{M}$ to be the normalization\footnote{See the Stacks project, \S 28.51 on relative normalization.} of $\mathcal{X}$ in $M$.
 By definition, $\mathcal{M}$ is a normal Deligne-Mumford stack sitting in  a commutative diagram
 \[
 \xymatrix{
 {   M   } \ar[r]\ar[d]  & { X } \ar[d] \\
 { \mathcal{M}  } \ar[r]  & {  \mathcal{X}  }.
 }
 \]
 The Kuga-Satake abelian scheme of \S \ref{ss:symplectic}  extends to an abelian scheme
 \[
 \mathcal{A}\to \mathcal{M}
 \]
  in an obvious way: simply pull back the universal object over $\mathcal{X}$.  
 
 \begin{theorem}[Kisin \cite{Kis}, Kim-Madapusi Pera \cite{KMP}]\label{thm:canonical integral}
 The integral model $\mathcal{M}$ is smooth over $\Z$, and is independent of the choice of $\psi_\delta$.  
 Moreover, there is a theory of automorphic vector bundles on $\mathcal{M}$, extending the theory in the generic fiber described in \S \ref{ss:bundles}. 
 \end{theorem}

 In particular, the final claim of the theorem provides us with filtered vector bundles $\bm{V}_\dR$ and $\bm{H}_\dR$ on $\mathcal{M}$  extending those  already constructed in the generic fiber, along with an inclusion
\[
\bm{V}_\dR \subset \underline{\End}( \bm{H}_\dR ).
\]
 A distinguished role will be played by the \emph{line bundle of weight one modular forms} 
 \[
\taut = F^1 \bm{V}_\dR .
 \]

\begin{remark}
 Theorem \ref{thm:canonical integral} remains true even if  the assumption $L^\vee = L$ is relaxed slightly.
 The important thing is that the compact open subgroup $K=\prod K_p$ be hyperspecial  at every prime $p$.  
 \begin{itemize}
 \item
 When $p>2$, $K_p$ is hyperspecial if and only if $L^\vee_p=L_p$.
\item
When $n$ is even, $K_2$ is hyperspecial if and only if $L_2^\vee=L_2$.
\item
When $n$ is odd the inclusion  $L_2\subset L_2^\vee$ is always proper, and $K_2$ is hyperspecial if and only if $[L^\vee_2:L_2]=2$.
\end{itemize}
Thus  our hypothesis $L^\vee=L$ implies that $n$ is even, and could have been  weakened to 
\[
[ L^\vee : L ] = \begin{cases}
1 &\mbox{if $n$ is even} \\
2 & \mbox{if $n$ is odd}.
\end{cases}
\]
We won't ever need this stronger version of Theorem \ref{thm:canonical integral}.
\end{remark}


\subsection{Special divisors in the self-dual case}
\label{ss:self-dual divisors}


Keep $L=L^\vee$ as above.

 Suppose $S$ is any scheme equipped with a morphism $S\to \mathcal{M}$,
and let $\mathcal{A}_S \to S$ be the pullback of the Kuga-Satake abelian scheme.   
As in \S \ref{ss:divisors}, one can define  a $\Z$-submodule
\[
V(\mathcal{A}_S) \subset \End(A_S)
\]
of \emph{special endomorphisms}, endowed with a positive definite $\Z$-valued quadratic form $Q$ satisfying
$x\circ x = Q(x) \cdot \mathrm{id}.$  The definition of a special endomorphism is now more subtle, as the Betti realization $x_B$ used in \S \ref{ss:divisors} to define $V(\mathcal{A}_S)$ when $S=\Spec(\C)$  is not available for arbitrary $S$.
Instead, one can use de Rham, $\ell$-adic, and crystalline cohomology in conjunction.

Briefly,  the vector bundle $\bm{H}_{\dR,S}$ on $S$ can be identified with the $\co_S$-linear dual of the first relative de Rham cohomology of $\mathcal{A}_S$, and so there is a canonical homomorphism
\[
\End(\mathcal{A}_S) \map{x \mapsto x_\dR} \End( \bm{H}_{\dR,S} )
\]
called \emph{de Rham realization}.   As we have already noted, (\ref{rep inclusion}) induces an inclusion of vector bundles
\[
\bm{V}_{\dR,S} \subset  \underline{\End}( \bm{H}_{\dR,S} ).
\]
We say that  $x\in \End(A_S)$ is \emph{de Rham special} if its de Rham realization   lies in the submodule
\[
H^0(S, \bm{V}_{\dR,S}) \subset  \End( \bm{H}_{\dR,S} ).
\]

Now fix a prime $\ell$.  The representations (\ref{two reps}) determine lisse \'etale sheaves $\bm{V}_{\ell,S[1/\ell]}$ and $\bm{H}_{\ell,S[1/\ell]}$ of $\Q_\ell$-modules over 
\[
S[1/\ell]=S\times_{\Spec(\Z)} \Spec(\Z[1/\ell] ),
\]
 related by an inclusion 
 \[
\bm{V}_{\ell ,S[1/\ell]} \subset  \underline{\End}( \bm{H}_{\ell,S[1/\ell]} ).
\]
The sheaf $ \bm{H}_{\ell,S[1/\ell]}$ is just the $\ell$-adic Tate module of $\mathcal{A}_{S[1/\ell]}$, and so there is an 
\emph{$\ell$-adic realization} map
\[
\End(\mathcal{A}_S) \map{x \mapsto x_\ell} \End( \bm{H}_{\ell,S[1/\ell]} ). 
\]
We say that  $x\in \End(\mathcal{A}_S)$ is \emph{$\ell$-special} if its $\ell$-adic realization $x_\ell$ lies in the submodule
\[
H^0(S, \bm{V}_{\ell,S[1/\ell]}) \subset  \End( \bm{H}_{\ell,S[1/\ell]} ).
\]

Now fix a prime $p$ and set $S_p=S\times_{\Spec(\Z)} \Spec(\F_p)$.  
 From the representations (\ref{two reps}) one  can  construct crystals of $\co_{S_p}$-modules $\bm{V}_{\crys,S_p}$ and $\bm{H}_{\crys,S_p}$, along with an inclusion 
\[
\bm{V}_{\crys,S_p} \subset \underline{\End}( \bm{H}_{\crys,S_p} ) 
\]
and a \emph{crystalline realization map}
\[
\End(\mathcal{A}_S) \map{x \mapsto x_\crys} \End( \bm{H}_{\crys,S_p} ).
\]
We then  define $x\in \End(\mathcal{A}_S)$ to be \emph{crystalline special} if its crystalline realization lies in the submodule
\[
H^0(S, \bm{V}_{\crys,S_p}) \subset  \End( \bm{H}_{\crys,S_p} )
\]
for every prime $p$.

Finally,  one defines  $x\in \End(\mathcal{A}_S)$ to be \emph{special} if it is de Rham special, $\ell$-special for every prime $\ell$, and crystalline special.
This definition may seem unwieldy, but it is simplified by the fact that specialness is an extremely rigid property.  
If $S$ is connected, and if there exists a geometric point $s\to S$ such that the restriction of $x: \mathcal{A}_S \to \mathcal{A}_S$ to the fiber $x_s: \mathcal{A}_s \to \mathcal{A}_s$ is special, then $x$ itself is special.
See \cite[\S 4.3]{AGHMP-2}.

\begin{definition}
For any positive integer $m$,  the \emph{special divisor}
\[
\mathcal{Z}(m) \to \mathcal{M}
\]
is the $\mathcal{M}$-stack with functor of points
\[
\mathcal{Z}(m)(S) = \{ x\in V(\mathcal{A}_S) : Q(x) =m \}
\]
for any $\mathcal{M}$-scheme $S\to \mathcal{M}$.
\end{definition}


\subsection{Integral models in the maximal case}
\label{ss:general integral}


We now drop the assumption $L^\vee=L$,  and return to the general case of a maximal lattice $L\subset V$.

Choose an isometric embedding of $V$ into a  quadratic space $V^\beef$ of signature $(n^\beef, 2)$, and 
 do this in such a way that $L$ is contained in a self-dual lattice $L^\beef \subset V^\beef$.
The isometric embedding $V\to V^\beef$ induces a homomorphism of Clifford algebras, which restricts to a morphism on the associated groups of spinor similitudes.  We obtain a morphism of the associated Shimura data, and hence a (finite and unramified) morphism $M\to M^\beef$ of the associated Shimura varieties over $\Q$.

The construction  of \S \ref{ss:smooth integral} provides us with a smooth integral model $\mathcal{M}^\beef$ of $M^\beef$, and a Kuga-Satake abelian scheme 
$
\mathcal{A}^\beef \to \mathcal{M}^\beef .
$ 
If we now define  $\mathcal{M}$ as the normalization of $\mathcal{M}^\beef$ in $M$, then $\mathcal{M}$ is a normal Deligne-Mumford stack, flat over $\Z$ with generic fiber $M$, sitting in a commutative diagram
\[
 \xymatrix{
 {  M  } \ar[r]\ar[d]  & {  M^\beef } \ar[d] \\
 {  \mathcal{M} } \ar[r]  & {\mathcal{M}^\beef  }.
 }
 \]
 The Kuga-Satake abelian scheme $A\to M$ extends uniquely to an abelian scheme 
 \[
 \mathcal{A} \to \mathcal{M}.
 \]

 Having already constructed a line bundle $\taut^\beef$ of weight one modular forms on $\mathcal{M}^\beef$, we can pull it back to $\mathcal{M}$ to obtain a line bundle $\taut$ on $\mathcal{M}$.
 There is no reason to expect that $\mathcal{M}$ admits any reasonable theory of automorphic vector bundles, extending the theory in the generic fiber, but the line bundle $\taut$ will suffice for our purposes.  
By endowing $\taut$ with the Petersson metric (\ref{pet metric}) in the complex fiber, we obtain a metrized line bundle
\begin{equation}\label{metrized taut}
\widehat{\taut} \in \widehat{\Pic}( \mathcal{M}).
\end{equation}
as in \S \ref{ss:arakelov}.

For a proof of the following, see \cite[Proposition 4.4.1]{AGHMP-2}.

\begin{proposition}
The integral model $\mathcal{M}$ and its line bundle $\taut$ do not  depend on the choice of $L^\beef$ used in their construction, and the Kuga-Satake abelian scheme $A\to M$ extends uniquely to an abelian scheme $\mathcal{A} \to\mathcal{M}$.
\end{proposition}


\subsection{Special divisors in the maximal case}
\label{ss:integral divisors}


Fix a   rational number $m>0$, a coset $\mu \in L^\vee/L$, and an   $\mathcal{M}$-scheme $S\to \mathcal{M}$.

Consider the Kuga-Satake abelian scheme
\[
\mathcal{A}^\beef \to \mathcal{M}^\beef
\] 
associated with the self-dual lattice $L^\beef$ used in the construction of $\mathcal{M}$.   
Viewing $S$ as an $\mathcal{M}^\beef$-scheme using $\mathcal{M} \to \mathcal{M}^\beef$,  we have already   defined a quadratic space of special endomorphisms $V(\mathcal{A}_S^\beef)$ in \S \ref{ss:self-dual divisors}
According to \cite[Proposition 2.5.1]{AGHMP-1} there is a canonical isometric embedding 
\[
\Lambda \to V(\mathcal{A}_S^\beef),
\]
where  $\Lambda = \{ \lambda \in L^\beef : \lambda \perp L \}$.  This allows us to define 
\begin{equation}\label{perp special}
V(\mathcal{A}_S) = \{ x\in V(\mathcal{A}_S^\beef) : x \perp \Lambda\}.
\end{equation}
Of course $V(\mathcal{A}_S)$ inherits from $V(\mathcal{A}_S^\beef)$ a positive definite quadratic form $Q$.

In fact, one can realize 
\[
V(\mathcal{A}_S) \subset \End(\mathcal{A}_S)
\]
in such way that $x\circ x = Q(x) \cdot \mathrm{id}$.
One first shows  that $\mathcal{A}$ comes equipped with a natural right action of the integral Clifford algebra $C(L)$, and similarly $\mathcal{A}^\beef$ comes with a right action of $C(L^\beef)$.  The isometric embedding $L\to L^\beef$ induces a ring homomorphism $C(L) \to C(L^\beef)$, and there is  a $C(L^\beef)$-linear isomorphism 
\begin{equation}\label{KS compare}
\mathcal{A}_S \otimes_{ C(L) } C(L^\beef) \iso \mathcal{A}^\beef_S ,
\end{equation}
 where the left hand side is Serre's tensor construction.
One can then identify (\ref{perp special}) with the $\Z$-module of all $C(L)$-linear endomorphisms of $\mathcal{A}_S$
 such that the induced endomorphism $x\otimes \mathrm{id}$ of (\ref{KS compare}) lies in $V(\mathcal{A}_S^\beef)$.
For all of this see \cite[\S 2.5]{AGHMP-1}.

Using the self-duality of $L^\beef$, one can easily check that the projections to the two factors of 
$
V \oplus \Lambda_\Q \iso   V^\beef
$
induce isomorphisms
\[
\xymatrix{
&{ L^\beef  / (L \oplus \Lambda)   } \ar[rd]\ar[dl]  \\
{ L^\vee / L }  & & { \Lambda^\vee / \Lambda },
}
\]
which allow us to view $\mu \in \Lambda^\vee /\Lambda$.  Using this and 
\[
V(\mathcal{A}_S)_\Q \oplus \Lambda_\Q \iso V(\mathcal{A}^\beef_S)_\Q,
\]
 we define
\[
V_\mu(\mathcal{A}_S) = \{ x \in V(\mathcal{A}_S)_\Q : x+\mu \in V(\mathcal{A}^\beef_S)\} .
\]
Taking $\mu=0$ recovers $V(\mathcal{A}_S)$.

\begin{definition}
For any rational number $m>0$ and any $\mu \in L^\vee/L$,  the \emph{special divisor}
\begin{equation}\label{special cycle}
\mathcal{Z}(m,\mu) \to \mathcal{M}
\end{equation}
is the $\mathcal{M}$-stack with functor of points
\[
\mathcal{Z}(m,\mu)(S) = \{ x\in V_\mu(\mathcal{A}_S) : Q(x) =m \}.
\]
\end{definition}

The stack $\mathcal{Z}(m,\mu)$ has dimension $\mathrm{dim}(\mathcal{M})-1$, and the map (\ref{special cycle}) is finite and unramified.  Its complex points can be identified with (\ref{eqn:Zmmu}) in a natural way.
Although (\ref{special cycle})  is not a closed immersion, one can construct from it  an effective Cartier divisor on $\mathcal{M}$, as explained in the discussion following Proposition \ref{prop:complex divisor}.
See also \cite[\S 2.7]{AGHMP-1} for more details.



\section{Regularized theta lifts and Borcherds products}
\label{s:harmonic}


Keep the notation of \S \ref{s:orthogonal shimura variety}.
Thus $(V,Q)$ is a quadratic space of signature $(n,2)$ with $n\ge 1$,   and $\mathcal{M}$ is the  orthogonal Shimura variety over $\Z$
determined by a maximal lattice $L\subset V$.


\subsection{Regularized theta lifts}
\label{ss:theta}


One can construct Green functions for certain linear combinations of special divisors $\mathcal{Z}(m,\mu)$ on $\mathcal{M}$ using the theory of regularized theta lifts.
The construction is due to the physicists Harvey and Moore, and was used by Borcherds in \cite{Bor98}  to simplify and extend the construction of Borcherds products first introduced in \cite{Bor95}.  The theory was then further extended by Bruinier \cite{Bruinier}.

Recall that the \emph{metaplectic double cover}
\[
\widetilde{\SL}_2(\Z)  \to \SL_2(\Z)
\]
is the group of all pairs $(g,\phi )$ in which 
\[g = \left(\begin{smallmatrix}  a & b \\ c & d \end{smallmatrix}\right) \in \SL_2(\Z),\]
 and $\phi: \mathcal{H} \to \C$ is a holomorphic function satisfying 
$\phi(\tau)^2 = c\tau+d$.  Multiplication is given by 
\[
\big( g_1,\phi_1(\tau) \big) \cdot \big(g_2,\phi_2(\tau) \big) = \big(g_1g_2, \phi_1(g_2 \tau) \phi_2(\tau) \big).
\]

The finite dimensional vector space $S_L= \C [L^\vee / L]$  of $\C$-valued functions on  $L^\vee / L$ admits an action 
\begin{equation}\label{little weil}
\weil_L: \widetilde{\SL}_2(\Z) \to \Aut(S_L)
\end{equation}
of the metaplectic double cover  $\SL_2(\Z)$, called the \emph{Weil representation}.
As in \cite[Remark 3.1.1]{AGHMP-1}, our $\weil_L$ is the complex conjugate of the representation $\rho_L$ used by Borcherds \cite{Bor98}. If $n$ is even, as it will be in our application to Colmez's conjecture, then $\weil_L$  factors through $\SL_2(\Z)$.

Suppose   $f : \mathcal{H} \to S_L$  is a holomorphic function satisfying
\[
f( g \tau ) = \phi(\tau)^{2-n} \weil_L(g,\phi)  f(\tau)
\]
for every $(g,\phi) \in \widetilde{\SL}_2(\Z)$ and $\tau\in \mathcal{H}$.  
Any such $f$ admits a Fourier expansion 
\[
f (\tau) = \sum_{ m \in \Q } c (m) q^{-m} , 
\]
with  coefficients  $c(m) \in S_L$, and each coefficient can be expanded uniquely as a linear combination 
\[
c(m) = \sum_{ \mu\in L^\vee  /  L }  c(m,\mu) \cdot \varphi_\mu,
\]
where  $\varphi_\mu\in S_L$ is the characteristic function of $\mu$.

If $f(\tau)$ is meromorphic  at $\infty$, in the sense that $c(m)$ vanishes   for all sufficiently negative $m$, 
we call  $f(\tau)$  a \emph{weakly holomorphic modular form}  of weight $1-\frac{n}{2}$ and representation $\weil_L$.
The infinite dimensional $\C$-vector space space of all weakly holomorphic modular forms of weight $1-\frac{n}{2}$ and representation $\weil_L$ is denoted $M^!_{1- \frac{n}{2} }(\weil_L)$.

It is a theorem of McGraw \cite{mcgraw} that $M^!_{1- \frac{n}{2}}(\weil_L)$ admits a $\C$-basis of forms  for which all $c(m,\mu)$ are integers, and we  fix one such  form
 \begin{equation}\label{input fourier}
f (\tau) = \sum_{ m \gg -\infty} c (m) q^{-m}  \in M^!_{1-\frac{n}{2}}(\weil_L).
\end{equation}
 If $c(m,\mu)\neq 0$ then  $m  + Q(\mu) \in \Z$, and  in particular  
 \[
m\in  \frac{1}{  [L^\vee : L] } \cdot  \Z .
 \]
It follows that  the  $\Z$-linear combination of special divisors
\[
\mathcal{Z}(f) = \sum_{m>0} \sum_{ \mu \in L^\vee / L} c(-m,\mu)  \cdot \mathcal{Z}(m,\mu) 
\]
is actually a finite sum, and so defines a Cartier divisor on $\mathcal{M}$.
We will  use the theory of regularized theta lifts to construct  a Green function $\Phi(f)$ for this divisor.

Suppose $z\in \mathcal{D}$, so that $z\in V_\C$ is an isotropic vector with $[z,\overline{z}]<0$.   
If we decompose $z = x+iy$ with $x,y\in V_\R$, then $x\perp y$ and $Q(x)=Q(y)$ is negative.
Let $P_z \subset V_\R$ be the oriented negative definite plane spanned by the ordered basis $x,y$.  It is an easy exercise to check that 
$z\mapsto P_z$ defines an isomorphism 
\[
\mathcal{D} \iso \{ \mbox{oriented negative definite planes in } V_\R \}
\]
of smooth manifolds.  Define 
\[
\mathrm{pr}_z  : V_\R \to V_\R ,\qquad   \mathrm{pr}_z^\perp : V_\R \to V_\R 
\]
to be the orthogonal projections to $P_z$ and $P_z^\perp$, respectively.

For every  $g\in G(\A_F)$ the associated lattice $L_g$ of (\ref{shifted lattice}) comes with an isomorphism
\[
 L^\vee / L \map{g\action} L_g^\vee / L_g,
\]
and hence every    $\varphi \in S_L$ determines a function $\varphi_g : L_g^\vee / L_g \to \C$.  Define the \emph{theta kernel}
\[
\theta_L    : \mathcal{H} \times \mathcal{D} \times G(\A_f)  \to S_L^\vee
\]
 by 
\[
\theta_L  ( \tau , z ,g  ,\varphi) = v  \sum_{  x \in L_g^\vee  }  \varphi_g(  x) 
 \cdot e^{2\pi i \tau \cdot Q( \mathrm{pr}_z^\perp (x)) } \cdot e^{2\pi i \overline{\tau} \cdot Q( \mathrm{pr}_z (x)) } .
\]

The theta kernel descends to a function
\[
\theta_L    : \mathcal{H} \times \mathcal{M}(\C)   \to S_L^\vee,
\]
which transforms in the variable $\tau \in \mathcal{H}$ like a modular form of weight  $\frac{n}{2} -1$ for the representation contragredient to (\ref{little weil}).  Using the tautological pairing $S_L \otimes S_L^\vee \to \C$, any weakly modular form (\ref{input fourier}) therefore defines a scalar-valued function 
\[
f \cdot \theta_L  : \SL_2(\Z) \backslash \mathcal{H} \times \mathcal{M}(\C) \to \C,
\]
and we may attempt to define a function
\[
\Phi(f , y) = \int_{\SL_2(\Z) \backslash \mathcal{H}   }  f (\tau) \cdot \theta_L(\tau,y)  \cdot \frac{du\,dv}{v^2}
\]
of the variable $y\in \mathcal{M}(\C)$ by integrating over the variable  $\tau=u+iv$ in the upper half-plane.

The above integral diverges due to the pole of $f(\tau)$ at the cusp $\infty$, but can be regularized as follows.  First define
\[
\Phi(f , y,s) = \lim_{T\to \infty} \int_{ F_T   }  f (\tau) \cdot \theta_L(\tau,y)  \cdot \frac{du\,dv}{v^{s+2}},
\]
where
\[
F_T = \{ \tau \in \mathcal{H} :   u^2+v^2 \ge 1  , \,  |u| \le \frac{1}{2}  ,  \,  v \le T \} 
\]
is the usual fundamental domain for $\SL_2(\Z) \backslash \mathcal{H} $ truncated at height $T$.
The limit exists for $\mathrm{Re}(s) \gg 0$, and admits meromorphic continuation to all $s$.

\begin{definition}
The \emph{regularized theta lift} of $f$ is the function on $\mathcal{M}(\C)$ defined by 
\[
\Phi(f , y) = \mbox{constant term of } \Phi(f , y,s) \mbox{ at } s=0.
\]
\end{definition}

The regularized theta lift $\Phi(f)$ is defined at every point of $\mathcal{M}(\C)$, but is discontinuous at points of the divisor $\mathcal{Z}(f)(\C)$.
Its values along  $\mathcal{Z}(f)(\C)$, while interesting, play no role in the present work.   They do play a fundamental role in the 
results of \cite{AGHMP-1} and \cite{BY09}.

\begin{theorem}[Borcherds \cite{Bor98}]\label{thm:green}
The regularized theta lift $\Phi(f)$ is a Green function for the divisor $\mathcal{Z}(f)$ on $\mathcal{M}$, and hence determines a class
\[
\widehat{\mathcal{Z}}(f) = \big( \mathcal{Z}(f) , \Phi(f) \big) \in \widehat{\mathrm{CH}}^1(\mathcal{M}).
\]
\end{theorem}

We will state a stronger result in Theorem \ref{thm:borcherds} below.


\subsection{Borcherds products}


We need to produce rational sections of powers of the line bundle $\taut$ on $\mathcal{M}$ with known Petersson norms (\ref{pet metric}).
Such sections were systematically constructed by Borcherds \cite{Bor95,Bor98,Bor:GKZ}.   

Fix an  
\[
f \in M_{1-\frac{n}{2}}^!(\weil_L),
\]
and assume that all coefficients $c(m,\mu)$  in the Fourier expansion (\ref{input fourier}) are integers.
After possible replacing $f$ by a positive integer multiple, there is an associated  \emph{Borcherds product} $\bm{\psi}(f)$.  This was initially constructed by Borcherds as a meromorphic function 
on the hermitian symmetric domain $\mathcal{D}$.  It is easy to construct from this function a meromorphic section of $\taut^{\otimes c(0,0)}$ on the complex fiber $\mathcal{M}(\C)$, and a calculation of Borcherds shows that the divisor of the meromorphic section is none other than the analytic divisor $\mathcal{Z}(f)(\C)$ on $\mathcal{M}(\C)$.

If the quadratic space $V$ contains an isotropic vector, the Shimura variety $\mathcal{M}(\C)$ is noncompact, and sections of $\taut$ and its powers have $q$-expansions.  Using the explicit $q$-expansion  computed by Borcherds, one can show that $\bm{\psi}(f)$ is algebraic and  defined over $\Q$, and that its divisor, when viewed as a rational section of $\taut^{\otimes c(0,0)}$ on the integral model, is $\mathcal{Z}(f)$.  
In fact, these same results are true even when $V$ contains no isotropic vector.  Of course the proofs are now more complicated, as  $\mathcal{M}(\C)$ is compact and no theory of $q$-expansions is available.

\begin{theorem}[Borcherds \cite{Bor98}, H.-Madapusi Pera \cite{HMP}]\label{thm:borcherds}
After possibly replacing $f$ by a positive integer multiple, 
there is a rational section $\bm{\psi}(f)$ of the line bundle $\taut^{  \otimes c(0,0) }$ on $\mathcal{M}$ such that 
\[
\mathrm{div}( \bm{\psi} (f) ) = \mathcal{Z}(f),
\] 
and such that
\[
\Phi(f)    =  -2  \log||  \bm{\psi}  (f)||  + c(0,0) \cdot \log(4\pi e^\gamma)  
\]
on $\mathcal{M}(\C) \smallsetminus \mathcal{Z}(f)(\C)$.
\end{theorem}

\begin{remark}
Our $\bm{\psi}(f)$ is equal to $(2\pi i )^{ c(0,0) }   \Psi(f)^{\otimes 2}$, where $\Psi(f)$ is the meromorphic section more commonly  referred to as the Borcherds product.
\end{remark}

\begin{remark}\label{rem:hormann}
When $V$ contains an isotropic vector and $L^\vee=L$,  Theorem \ref{thm:borcherds} is contained in the work of H\"ormann \cite{hor:thesis} and \cite{hor:book}.
\end{remark}

\begin{corollary}\label{cor:taut to special}
Inside the codimension one arithmetic Chow group 
\[
\widehat{\Pic}( \mathcal{M}) \iso \widehat{\mathrm{CH}}^1(\mathcal{M})
\]
 we have the equality
\[
  c(0,0)    \cdot \widehat{\taut}  =   \widehat{\mathcal{Z}}(f)  -   c(0,0) \cdot \big( 0 , \log ( 4\pi e^\gamma ) \big) ,
\]
where $(0 , \log ( 4\pi e^\gamma ))$ means the trivial divisor on $\mathcal{M}$ endowed with the constant Green function $\log(4\pi e^\gamma)$ on the complex fiber.
\end{corollary}


\section{The big CM cycle}
\label{s:big cm}


We are now going to switch the point of view slightly from previous sections.  
Rather than starting with a quadratic space  as our initial data,  we will start with  a CM field $E$ of degree $2d>2$, and build from it a quadratic space $(V,Q)$ of signature $(2d-2,2)$ in such a way that 
the resulting  $G=\GSpin(V)$ comes with a distinguished maximal torus $T\subset G$.  

Geometrically, this maximal torus  corresponds to a $0$-cycle $Y  \to M$ of points at which the Kuga-Satake abelian scheme is isogenous to a power of an abelian variety with complex multiplication by the total reflex algebra $E^\sharp$ defined in  \S \ref{ss:reflex}.


\subsection{The initial data}


Let $E$ be a CM field of degree $2d>2$,  and let $F$ be its maximal totally real subfield.  
Denote by 
\[
\iota_0,\ldots, \iota_{d-1}  :  F \to \R
\] 
the real embeddings of $F$, and choose a  $\xi \in F^\times$ such that $\iota_0(\xi)$ is negative, while $\iota_1(\xi),\ldots, \iota_{d-1}(\xi)$ are  positive.
The  rank two quadratic space
\[
( \mathscr{V} , \mathscr{Q} )  = ( E , \xi \cdot \mathrm{Norm}_{E/F} )
\]
over $F$ has signature 
\[
\mathrm{sig}( \mathscr{V} , \mathscr{Q} ) = \big( (0,2), (2,0), \ldots, (2,0)  \big).
\]
That is, the quadratic form is negative definite at $\iota_0$, and positive definite at the remaining real places.

 Define a quadratic space   $V$ over $\Q$ by 
\begin{equation}\label{trace quadratic}
(V,Q)=( \mathscr{V} , \mathrm{Trace}_{F/\Q} \circ \mathscr{Q})
\end{equation}
of signature $(2d-2,2)$.  As in \S \ref{ss:gspin}, let $C=C^+\oplus C^-$ be its Clifford algebra.

One can  similarly form the Clifford algebra of the $F$-quadratic space $\mathscr{V}$.  
It is a $\Z/2\Z$-graded quaternion algebra over $F$.
The even part  is isomorphic to $E$, while the odd part is simply $\mathscr{V}$ itself.
Obviously the relation (\ref{trace quadratic}) should imply some relation between the Clifford algebras of $V$ and $\mathscr{V}$. 
The  relation is slightly subtle, and involves the  total reflex algebra  $E^\sharp$ of  \S \ref{ss:reflex}.

\begin{proposition}\label{prop:reflex embedding}
There is a canonical injection of $\Q$-algebras $E^\sharp \to C^+$, which makes $C$ into a free $E^\sharp$-module of rank 
$2^d$.
\end{proposition}

\begin{proof}
Choose a totally real number field $\kk$ that is Galois over $\Q$, and for any $\Q$-vector space $W$ set $W_\kk=W\otimes_\Q\kk$.
We may choose $\kk$ in such a way that 
\[
F_\kk  \iso \underbrace{\kk \oplus \cdots \oplus \kk}_{ d\mbox{ times}}.
\]
This implies that 
$
E_\kk  \iso \kk_1\oplus \cdots \oplus \kk_d
$
for CM fields $\kk_1,\ldots, \kk_d$, each having $\kk$ as its maximal totally real subfield.
An exercise in Galois theory shows that 
\[
E_\kk^\sharp  \iso \kk_1\otimes_\kk \cdots \otimes_\kk \kk_d
\]
as $\kk$-algebras.

The action of $F$ on $V=\mathscr{V}$ induces  an orthogonal decomposition of $\kk$-quadratic spaces
$
V_\kk \iso W_1 \oplus \cdots \oplus W_d ,
$
in which each $W_i$ is a $1$-dimensional vector space over $\kk_i$,
and  a corresponding isomorphism of $\kk$-algebras
\[
C_\kk  \iso D_1 \otimes_\kk \cdots \otimes_\kk D_d,
\]
in which  $D_i \iso \kk_i \oplus W_i$ is the Clifford  algebra of the $2$-dimensional quadratic space $W_i$ over $\kk$.  

In particular, we have a natural inclusion $E_\kk^\sharp \hookrightarrow C_\kk$.  
One can check that this is compatible with the $\Gal(\kk/\Q)$ action on source and target, and  descends to an injection $E^\sharp \to C$ with the desired properties.
\end{proof}


\subsection{A maximal torus}
\label{ss:groups}


The relation (\ref{trace quadratic}) endows the spinor similitude group $G=\GSpin(V)$ with extra structure, namely a distinguished maximal torus $T\subset G$, which we now describe.

First define tori over $\Q$ by 
\[
T_F = \mathrm{Res}_{F/\Q} \mathbb{G}_m,\qquad 
T_E = \mathrm{Res}_{E/\Q} \mathbb{G}_m 
\]
and
\[
T_{F^\sharp} = \mathrm{Res}_{F^\sharp/\Q} \mathbb{G}_m , 
\qquad T_{E^\sharp} = \mathrm{Res}_{E^\sharp/\Q} \mathbb{G}_m,
\]
where  $E^\sharp$ is the total reflex algebra of $E$,  and $F^\sharp$ is its maximal totally real subalgebra.  
Define a quotient torus
\[
T = T_E / \mathrm{ker}\big( \mathrm{Norm}_{F/\Q} : T_F \to \mathbb{G}_m \big),
\]
and a  subtorus $T^\sharp\subset T_{E^\sharp}$  by the cartesian diagram
\[
\xymatrix{
{ T^\sharp  }   \ar[rr] \ar[d] & & {  \mathbb{G}_m    } \ar[d]  \\
{ T_{E^\sharp} }   \ar[rr]_{\mathrm{Norm}_{ E^\sharp/F^\sharp } } &&  {    T_{F^\sharp}    }  
}
\]
where the vertical arrow on the right is induced by the inclusion $\Q  \subset F^\sharp$.

The total reflex norm of Proposition \ref{prop:reflex pair} induces a morphism
\[
\mathrm{Nm}^\sharp : T_E \to T_{E^\sharp}, 
\]
which takes values in the subgroup $T^\sharp$.  Moreover, this morphism
 restricts to $\mathrm{Norm}_{F/\Q}$  on the subgroup $T_F \subset T_E$, and so factors through a morphism
 \begin{equation}\label{better norm}
 \mathrm{Nm}^\sharp : T \to T^\sharp .
 \end{equation}

 Let $\psi_\delta$ be a symplectic form on the $\Q$-vector space $H=C$,  of the type constructed in  \S \ref{ss:symplectic}.
Using Proposition \ref{prop:reflex embedding}, we can view $H$ as a left module over   $E^\sharp  \subset C$, and hence we obtain a faithful representation
\[
T_{E^\sharp} \to \GL(H).
\]
It is not hard to see that this restricts to a morphism
\begin{equation}\label{hodge torus}
T^\sharp \to \GSp(H),
\end{equation}
and there is a cartesian diagram of algebraic groups
\[
\xymatrix{
{  T   } \ar[rr] \ar[d]_{(\ref{better norm})}  && G \ar[d] \\
{ T^\sharp}  \ar[rr]_{ (\ref{hodge torus}) }  && \GSp(H)
}
\]
over $\Q$.  Here the vertical arrow on the right is the symplectic representation of \S \ref{ss:symplectic}, and 
the top horizontal arrow is uniquely determined by the commutativity of the diagram.


\subsection{Shimura data}


Now we attach a  $0$-dimensional Shimura variety  to the maximal torus $T\subset G$  defined in  \S \ref{ss:groups}.

Recall that we have a distinguished  real embedding $\iota_0 : F \to \R$, characterized as the unique real place of $F$ at which the quadratic space $\mathscr{V}$ is negative definite.  In other words, by the condition that $\mathscr{V} \otimes_{F,\iota_0} \R$
is a real quadratic space of signature $(0,2)$.   Fix an extension to 
\[
\iota_0:E \to \C.
\]

Recall that  $V=\mathscr{V}$ has the structure of a $1$-dimension vector space over $E$. This structure 
determines a decomposition
\[
V_\C = \bigoplus_{ \iota : E\to \C} V(\iota)
\] 
into complex lines, where  
\[
V(\iota) = \{ x\in V_\C : \forall  \beta\in E,\ \beta x = \iota(\beta) x   \} .
\] 
In particular, we obtain a distinguished  line $V(\iota_0) \subset V_\C$.  It is an easy exercise in linear algebra to check that this line is isotropic, and defines a point 
\[
z_0 \in \mathcal{D} \subset \Hom(\mathbb{S} , G_\R)
\]
in the hermitian symmetric domain (\ref{domain}).  
Moreover,   the  corresponding morphism $\mathbb{S} \to G_\R$  takes values in $T_\R$.

The pair $(T,\{ z_0\})$ is a Shimura datum with  reflex field  $\iota_0(E) \subset \C$, and the inclusion $T\subset G$ defines a morphism of Shimura data
\begin{equation}\label{big data}
(T  ,\{ z_0\}) \to (G,\mathcal{D}) .
\end{equation}


\subsection{The big CM cycle}
\label{ss:big cm}


Fix a maximal lattice $L\subset V$.  
As in \S \ref{ss:the shimura variety}, this choice determines a compact open subgroup   $K\subset G(\A_f)$, and an orthogonal  Shimura variety 
$M$ over $\Q$.    Let $\mathcal{M}$ be the integral model defined in \S \ref{ss:general integral}.

Fix any compact open subgroup  $K_T\subset T(\A_f)$  small enough that
\[
K_T \subset K\cap T(\A_f) .
\] 
The canonical model $Y$ of the    $0$-dimensional Shimura variety
\begin{equation}\label{complex big cm}
Y   (\C) = T(\Q) \backslash \{ z_0\} \times T(\A_f) / K_T
\end{equation}
 is a reduced Deligne-Mumford stack $Y$, finite  over the reflex field $\iota_0(E)$.  
We can define an integral model $\mathcal{Y}$  simply by taking the normalization of  $\Spec( \iota_0(\co_E))$ in $Y$.  
 Thus $\mathcal{Y}$ is a regular Deligne-Mumford stack, finite and flat over $\iota_0(\co_E)$ with generic fiber $Y$.

The natural morphism $Y(\C)  \to M(\C)$ induced by (\ref{big data}) is algebraic, and descends to the subfield 
$\iota_0(E) \subset \C$.  One can show that it extends uniquely to a morphism 
\begin{equation}\label{almost big cm}
\mathcal{Y} \to \mathcal{M} \times_{\Spec(\Z)} \Spec( \iota_0(\co_E)  )
\end{equation}
of integral models over $\iota_0(\co_E)$.

The notation $\mathcal{Y}(\C)$ is potentially ambiguous, as it could refer either to the complex points of $\mathcal{Y}$ viewed as a stack over $\Z$, or the complex points of $\mathcal{Y}$ viewed as a stack over $\iota_0(\co_E)$.   
To disambiguate, we understand that  $\mathcal{Y}(\C)$ always means the latter.  In other words $\mathcal{Y}(\C)$ is the same as (\ref{complex big cm}).

We define $\mathcal{Y}_\Z = \mathcal{Y}$, but viewed as a stack over $\Z$ rather than $\iota_0(\co_E)$.
Thus the stacks $\mathcal{Y}_\Z$ and  $\mathcal{Y}$ are the same, but have different complex points.  
More precisely, 
 \[
 \mathcal{Y}_\Z(\C) \iso \bigsqcup_{ \sigma : \iota_0(\co_E) \to \C } \mathcal{Y}^\sigma(\C).
 \]
Composing (\ref{almost big cm}) with the projection  
\[
 \mathcal{M} \times_{\Spec(\Z)} \Spec( \iota_0(\co_E)  )   \to \mathcal{M} ,
\]
 we obtain a morphism of $\Z$-stacks 
\[
\mathcal{Y}_\Z  \to \mathcal{M} 
\]
called  the \emph{big CM cycle}.
As in \S \ref{ss:arakelov}, there is an induced  homomorphism
\begin{equation}\label{big cm functional}
 [- : \mathcal{Y}_\Z ] : \widehat{\mathrm{CH}}^1(\mathcal{M}) \to \R
\end{equation}
called \emph{arithmetic intersection against $\mathcal{Y}_\Z$}.

This linear functional depends on the  choice of compact open subgroup $K_T \subset T(\A_f)$ used to define $\mathcal{Y}_\Z$, 
but if we set 
\[
\deg_\C( \mathcal{Y}_\Z) = \sum_{y\in \mathcal{Y}_\Z(\C)} \frac{1}{ |\Aut(y)| }.
\]
then  the rescaled linear functional
\[
\widehat{\mathcal{Z}} \mapsto \frac{ [\widehat{\mathcal{Z}} : \mathcal{Y}_\Z ] }{ \deg_\C(\mathcal{Y}_\Z) }
\]
is independent of the choice.

At a point in the image of $\mathcal{Y}_\Z(\C) \to \mathcal{M}(\C)$  the fiber of the  Kuga-Satake abelian scheme is isogenous to a power of  an abelian variety with complex multiplication.
More precisely, we have the following proposition.

\begin{proposition}\label{prop:KScm}
The pullback of the  Kuga-Satake abelian scheme 
\[
\mathcal{A}\to \mathcal{M}
\]
 to any complex point  $y\in \mathcal{Y}_\Z(\C)$ is isogenous to $2^d$ copies of an abelian variety over $\C$ with complex multiplication by $E^\sharp$, and CM type a Galois conjugate of $\Phi^\sharp$.
\end{proposition}

\begin{proof}
First suppose $y\in \mathcal{Y}(\C) \subset \mathcal{Y}_\Z(\C)$, 
so that $y=[ z_0 , g]$ for some $g\in T(\A_f) \subset G(\A_f)$.  
The fiber of the Kuga-Satake abelian scheme at $y$ is then given by 
\[
\mathcal{A}_y(\C) \iso (g H_{\widehat{\Z}}   \cap H)  \backslash H_\C /  z_0 H_\C,
\]
where we are identifying $z_0\in  V_\C  \subset  \End(H_\C)$ using the left multiplication action of $V\subset C$ on $C=H$.

We know from Proposition \ref{prop:reflex embedding} that left multiplication by the subalgebra $E^\sharp \subset C^+$ makes  $H$ into a free $E^\sharp$-module of rank $2^d$. 
 It is an exercise in linear algebra to check that the subspace $z_0 H_\C \subset  H_\C$ is stable under the action of $E^\sharp$.  This subspace is obviously stable under the right multiplication action of $C$, and hence we obtain an action
\[
M_{2^d}(E^\sharp) \iso E^\sharp \otimes_\Q C \to \End(\mathcal{A}_y)_\Q.
\]
One can then use this action to decompose $\mathcal{A}_y$ as a product, up to isogeny, of $2^d$ factors.
Each factors admits complex multiplication by $E^\sharp$, and all factors  are $E^\sharp$-linearly isogenous to one another.
Another exercise in linear algebra shows that the eigenspace of a $\Q$-algebra map $\varphi \in \Hom(E^\sharp , \C)$ in 
\[
\Lie(\mathcal{A}_y) \iso H_\C /  z_0 H_\C
\]
has dimension $2^d$ if $\varphi\in \Phi^\sharp$, and has dimension $0$ otherwise.  This implies that each of the isogeny factors of $\mathcal{A}_y$ has complex multiplication of type $(E^\sharp , \Phi^\sharp)$.

A general point $y\in \mathcal{Y}_\Z(\C)$ is always $\Aut(\C/\Q)$-conjugate to a point of $\mathcal{Y}(\C)$, and hence decomposes, up to isogeny, as a product of abelian varieties with CM by $E^\sharp$ and CM type a Galois conjugate of $\Phi^\sharp$.
\end{proof}


\subsection{The average Faltings height}


We now relate $[\widehat{\taut} : \mathcal{Y}_\Z]$ to the average Faltings height appearing in  Theorem \ref{thm:average colmez}.   Before stating the result, we need a definition. 
 If  $p$ is any rational prime, the action of $E$ on $V=\mathscr{V}$ induces a decomposition
 \[
 V_p = \bigoplus_{\mathfrak{p} \mid p} \mathscr{V}_\mathfrak{p}
 \]
 over the primes  $\mathfrak{p} \subset \co_F$ above $p$. For any such  $\mathfrak{p}$  set
 \[
L_\mathfrak{p} = L_p \cap \mathscr{V}_\mathfrak{p},
\]
where the intersection is taken inside of $V_p$.  The  quadratic form $Q$ on $L_p$ restricts to a $\Z_p$-valued quadratic form on $L_\mathfrak{p}$.

\begin{definition}\label{def:bad primes}
We say that a rational prime $p$ is \emph{good} if it satisfies the following properties.
\begin{enumerate}
\item  
For every prime $\mathfrak{p} \subset \co_F$ above $p$ that is unramified in $E$, the $\Z_p$-lattice $L_\mathfrak{p}\subset \mathscr{V}_\mathfrak{p}$ is  self-dual with respect to the  $\Z_p$-valued quadratic form, and  is stable under the action of $\co_{E,\mathfrak{p}}$.
\item
For every prime $\mathfrak{p} \subset \co_F$ above $p$ that is ramified in $E$, the $\Z_p$-lattice $L_\mathfrak{p}\subset \mathscr{V}_\mathfrak{p}$ is  maximal with respect to the  $\Z_p$-valued quadratic form, and there exists an $\co_{E,\mathfrak{q}}$-stable lattice
$
\Lambda_\mathfrak{p} \subset \mathscr{V}_\mathfrak{p}
$
such that 
\[
\Lambda_\mathfrak{p} \subset L_\mathfrak{p} \subset  \mathfrak{d}^{-1}_{E_\mathfrak{q} / F_\mathfrak{p}} \Lambda_\mathfrak{p}.
\]
Here $\mathfrak{q} \subset \co_E$ is the unique prime above $\mathfrak{p}$, and $\mathfrak{d}^{-1}_{E_\mathfrak{q} / F_\mathfrak{p}} \subset \co_{E,\mathfrak{q}}$ is the inverse  different of $E_\mathfrak{q} / F_\mathfrak{p}$.
\end{enumerate}
Denote by $\Sigma_\mathrm{bad}$ the (finite) set of rational primes that are not good.
\end{definition}

\begin{theorem}\label{thm:taut to faltings}
The equality
\begin{equation}\label{taut computes average}
 \frac{    [  \widehat{\taut}   :    \mathcal{Y}_\Z   ]   }{  \deg_\C(\mathcal{Y}_\Z)   }    = \frac{1}{ d \cdot 2^{d-1} }\sum_\Phi h^\Falt_{(E,\Phi)}  + \log(2\pi)
\end{equation}
holds up to a $\Q$-linear combination of $\{ \log(p) :  p \in \Sigma_\mathrm{bad} \}$.
\end{theorem}

\begin{proof}
We will not give a complete proof of this result.  
Instead, we will first prove that (\ref{taut computes average}) holds up to a $\Q$-linear combination of $\{\log(p) : p \mbox{ prime}\}$, and then 
explain how to extract the stronger statement from \cite{AGHMP-2}.

Recalling the notation $\circeq$ of the introduction, 
if we choose any global section $s$ of $\taut|_{\mathcal{Y}_\Z}$ then, directly from the definition of (\ref{big cm functional}),
\[
 \frac{    [  \widehat{\taut}   :    \mathcal{Y}_\Z   ]   }{  \deg_\C(\mathcal{Y}_\Z)   } 
 \circeq 
   \frac{ - 1   }{  \deg_\C(\mathcal{Y}_\Z)   } 
  \sum_{ y\in \mathcal{Y}_\Z(\C) }  \frac{  \log || s_y|| }{ |\Aut(y)| }.
 \]
Using Theorem \ref{thm:modular height}, one can rewrite this as  
 \begin{align*}
  \frac{    [  \widehat{\taut}   :    \mathcal{Y}_\Z   ]   }{  \deg_\C(\mathcal{Y}_\Z)   }  
& \circeq
  \frac{ - 1   }{  \deg_\C(\mathcal{Y}_\Z)   }   
   \sum_{ y\in \mathcal{Y}_\Z(\kk) }    \frac{  1 }{ |\Aut(y)| }
  \sum_{ \sigma : \kk \to \C}   \frac{\log || s_{y^\sigma} ||  }{ [ \kk : \Q]  } \\
&  \circeq  
  \frac{ 1   }{  \deg_\C(\mathcal{Y}_\Z)   }   
     \sum_{ y\in \mathcal{Y}_\Z(\kk) } 
    \frac{h^\Falt ( A_y) }{ 2^{2d-2}   |\Aut(y)|   }    +   \log(2\pi) 
 \end{align*}
for any number field $\kk \subset \C$ Galois over $\Q$ and large enough that $\mathcal{Y}_\Z(\kk) = \mathcal{Y}_\Z(\C)$.
Now combine this with Proposition \ref{prop:KScm} and  Corollary \ref{cor:reflex height}, which  imply
\[
   \frac{1}{ 2^{2d-2} }  \cdot h^\Falt ( A_y)  
  \circeq 
   \frac{1}{ 2^{d-2} }  \cdot h^\Falt_{(E^\sharp, \Phi^\sharp)}   
    =
   \frac{1}{ d \cdot 2^{d-1}   }  \cdot    \sum_{ \Phi } h^\Falt_{( E,\Phi) } .
\]

We explain how to extract Theorem \ref{thm:taut to faltings} from the results of \cite{AGHMP-2}.
In   \cite[\S 9]{AGHMP-2} there is another stack $\mathcal{Y}_0$, regular and finite  flat over $\iota(\co_E)$, equipped with a finite morphism $\mathcal{Y} \to \mathcal{Y}_0$.  On this stack there is a metrized line bundle $\widehat{\taut}_0$.   Theorem 9.4.2 of \cite{AGHMP-2} asserts that
\[
 \frac{    \widehat{\deg}(   \widehat{\taut}_0 ) }{  \deg_\C(\mathcal{Y}_{0\Z})   }  +\frac{1}{2d} \cdot \log|D_F|  = \frac{1}{ d \cdot 2^{d-1} }\sum_\Phi h^\Falt_{(E,\Phi)}  + \log(2\pi)
\]
where $\mathcal{Y}_{0\Z}$ is $\mathcal{Y}_0$  regarded as a stack over $\Z$, and 
\[
\widehat{\deg} : \widehat{\mathrm{CH}}^1( \mathcal{Y}_{0\Z} ) \to \R
\]
is the arithmetic degree of \S \ref{ss:arakelov}.
It is shown in the proof of Proposition 9.5.1 of  \cite{AGHMP-2}  that 
\[
\frac{  2d\cdot   [  \widehat{\taut}   :    \mathcal{Y}_\Z   ]   }{  \deg_\C(\mathcal{Y}_\Z)   }  
 =   \frac{  2d\cdot\widehat{\deg} (\widehat{\taut}_0) }{ \deg_\C( \mathcal{Y}_{0\Z})  } +   \log|D_F| 
\]
 holds up to a $\Q$-linear combination of $\{ \log(p) :  p \in \Sigma_\mathrm{bad} \}$.
 Note that the factors  $2d/\deg_\C(\cdot)$ were mistakenly omitted from equation (9.5.1) of 
 \cite{AGHMP-2}.
\end{proof}


\section{The arithmetic intersection formula}
\label{s:special to eisenstein}


We keep the data of \S \ref{s:big cm}.  
In particular,  $E$ is a CM field of degree $2d>2$ with maximal totally real subfield $F$.

Recall that we have fixed an  embedding $\iota_0 : F \to \R$, and an extension of it to a complex embedding of $E$.
We chose an element $\xi\in F^\times$ that is negative at $\iota_0$, and positive at all other archimedean places, and  constructed from it a quadratic space $(V,Q)$ of signature $(2d-2,2)$.
A choice of maximal lattice $L\subset V$ then determined an orthogonal Shimura variety $\mathcal{M}$ endowed with a metrized line bundle $\widehat{\taut}$  and a big CM cycle 
\[
\mathcal{Y}_\Z \to \mathcal{M}.
\]

We saw  (Theorem \ref{thm:taut to faltings}) that the arithmetic intersection $[\widehat{\taut} : \mathcal{Y}_\Z ]$ is essentially the averaged Faltings height appearing in Theorem \ref{thm:average colmez}.
On the other hand, we have seen (Corollary \ref{cor:taut to special})  that the theory of Borcherds products relates  the metrized line bundle $\widehat{\taut}$ to the arithmetic divisor $\widehat{\mathcal{Z}}(f)$, where $f$ is a suitable (vector-valued) weakly holomorphic form  on the complex upper half-plane.

To prove  Theorem \ref{thm:average colmez}, it remains to  relate the arithmetic intersection $[\widehat{\mathcal{Z}}(f) : \mathcal{Y}_\Z]$ to the logarithmic derivative at $s=0$ of $L(s,\chi)$.  
This will be done by relating both quantities to the constant term of an Eisenstein series.  
These calculations are the technical core of \cite{AGHMP-2}.  

At this point in the exposition we abandon any pretense of providing complete proofs.
We  hope only to highlight to the reader some of the  main ideas, and provide references to the appropriate places in  \cite{AGHMP-2} where proofs can be found.


\subsection{An incoherent Eisenstein series}


We now recall from the work of Bruinier-Kudla-Yang \cite{BKY} a very particular Eisenstein series on the adelic group $\SL_2(\A_F)$.
We will follow the exposition of \S 6.1 and \S 6.2  of  \cite{AGHMP-2}, although everything we say can be found in  \cite{KuAnnals} in greater generality.

Recall that our initial data in the construction of the big CM cycle 
\[
\mathcal{Y}_\Z \to \mathcal{M}
\]
 is a quadratic space $(\mathscr{V},\mathscr{Q})$ over $F$ of rank two, which is negative definite at $\iota_0 : F \to \R$, but positive definite at all other archimedean places.  
 Define a rank two quadratic space 
$(\mathscr{C},\mathscr{Q})$ over the adele ring $\A_F$ by declaring that $\mathscr{C}$ is positive definite at \emph{every} archimedean place of $F$, and that $\mathscr{C}_\mathfrak{p} \iso \mathscr{V}_\mathfrak{p}$ as $F_\mathfrak{p}$ quadratic spaces for every finite place $\mathfrak{p}$ of $F$.  
Thus $\mathscr{C}$ and $\mathscr{V}\otimes_F\A_F$ are isomorphic everywhere locally except at the archimedean place $\iota_0$.

\begin{remark}
In the terminology of \cite{KuAnnals}, the quadratic space $\mathscr{C}$ over $\A_F$ is \emph{incoherent} in the sense that it does not arise as the adelization of any quadratic space over $F$.     
\end{remark}

Let $\mathscr{C}_\infty$ be the product of the archimedean components of $\mathscr{C}$, and identify the finite part of $\mathscr{C}$ with $\widehat{\mathscr{V}} = \mathscr{V} \otimes_F \widehat{F}$, so that 
\[
\mathscr{C}\iso \mathscr{C}_\infty\times  \widehat{ \mathscr{V} }.
\]
For each $\mu \in L^\vee / L$ the space  
\[
S(\mathscr{C}) \iso S(\mathscr{C}_\infty) \otimes S( \widehat{ \mathscr{V} } )
\]
of Schwartz functions on $\mathscr{C}$ contains a distinguished element $\varphi_\infty^{\bm{1}} \otimes \varphi_\mu$.
The  archimedean component 
$
\varphi^{\bm{1}}_\infty \in S(\mathscr{C}_\infty)
$
is defined as the product over all archimedean places $v$ of $F$ of the Gaussian distributions
\[
\varphi^{\bm{1}}_v( x) = e^{-2\pi \mathscr{Q}_v(x) } \in S(\mathscr{C}_v),
\]
where $\mathscr{Q}_v$ is the quadratic form on $\mathscr{C}_v$.  The finite component 
$
\varphi_\mu \in S(\widehat{\mathscr{V}})
$
is the characteristic function of the compact open subset
$
\mu + \widehat{L} \subset \widehat{\mathscr{V}}.
$

The group $\SL_2(\A_F)$ acts\footnote{There is no need to work with the metaplectic group, as $\mathrm{rank}_{\A_F}( \mathscr{C})=2$ is even.} on $S(\mathscr{C})$ via the Weil representation $\weil_\mathscr{C}$.  Using this representation to translate the Schwartz function $\varphi_\infty^{\bm{1}} \otimes \varphi_\mu$,   we define
\begin{equation}\label{standard section}
\Phi(g,s,\varphi_\mu) = \weil_\mathscr{C}(g) (\varphi^{\bm{1}}_\infty\otimes\varphi_\mu )(0)  \cdot | a |^s
\end{equation}
for $g\in \SL_2(\A_F)$,  $s\in \C$, and $\mu\in L^\vee /L$.
Here $|a|^s$ is defined by factoring
\[
g = \left(\begin{matrix} 1 & b \\  0 & 1   \end{matrix}  \right) \left(\begin{matrix} a & 0 \\  0 & a^{-1}   \end{matrix}  \right)  k
\]
with $k\in \SL_2(\widehat{\co}_F)$.

The function (\ref{standard section})  satisfies the transformation law
\[
\Phi\left(\left( \begin{matrix} a &  b \\ & a^{-1} \end{matrix} \right) g , s ,\varphi_\mu \right)  = \chi( a) |a|^{s+1} \Phi(g,s, \varphi_\mu)
\]
for all $a\in \A_F^\times$ and $b\in \A_F$.  In other words, it lies in the space $I( s, \chi )$ of the representation obtained by inducing the character 
$\chi  |\cdot |^s$ from the standard Borel subgroup of $\SL_2(\A_F)$.  Thus we may form the Eisenstein series
\[
E(g,s,\varphi_\mu) = \sum_{   \gamma \in B(F) \backslash \SL_2(F)  } \Phi(\gamma g ,s, \varphi_\mu), 
\]
where $B\subset \SL_2$ is the group of upper triangular matrices.

Our ordering of the real embeddings $\iota_0,\ldots,\iota_{d-1}  : F \to \R$ fixes an isomorphism $F\otimes_\Q \C \iso \C^d$, and we define  
\[
\mathcal{H}_F \subset F \otimes_\Q \C
\]
to be the preimage of $\mathcal{H}^d\subset \C^d$ under this isomorphism.

We now de-adelize the automorphic form $E(g,s,\varphi_\mu)$ to obtain a classical Hilbert modular Eisenstein series 
in the variable  $\vec{\tau} \in \mathcal{H}_F$.   Writing $\vec{\tau} = \vec{u} + i \vec{v}$ with $\vec{u},\vec{v} \in F \otimes_\Q \R$ and $\vec{v}$ totally positive,  we define
\[
g_{ \vec{\tau} } = \left( \begin{matrix}  1 & \vec{u} \\ & 1 \end{matrix} \right) \left( \begin{matrix}   \sqrt{ \vec{v} }  & \\ &  1/ \sqrt{ \vec{v}}   \end{matrix} \right)  \in \SL_2( F\otimes_\Q\R),
\]
and view $g_{ \vec{\tau} }  \in \SL_2(\A_F)$ with trivial non-archimedean components.  
The function  
\[
E(\vec{\tau} , s , \varphi_\mu) =  \frac{1}{ \sqrt{ \mathrm{Norm}_{F/\Q} (\vec{v}) } } \cdot  E(g_{\vec{\tau}} ,s,\varphi_\mu)
\]
is a nonholomorphic Hilbert modular form of parallel weight $1$, and satisfies a functional equation in $s\mapsto -s$.
Moreover, this Eisenstein series vanishes identically at the center $s=0$ of the functional equation.

\begin{remark}
  The claim about the weight of $E(\vec{\tau} , s , \varphi_\mu)$  follows from the particular choice of Schwartz function $\varphi_\infty^{\bm{1}}$, and is the reason for the otherwise mysterious superscript $\bm{1}$.  The claim about the vanishing at $s=0$ follows from the incoherence of the quadratic space $\mathscr{C}$.
\end{remark}

Consider the central derivative
\[
E'(\vec{\tau} , 0 , \varphi_\mu) = \frac{d}{ds} E(\vec{\tau} , s , \varphi_\mu) \big|_{s=0} .
\]
Like any Hilbert modular form it admits a $q$-expansion, but because it is nonholomorphic the coefficients depend on the imaginary part \[\vec{v} \in F\otimes_\Q\R\]  of $\vec{\tau}$.
  We will apply what one might call the lazy man's holomorphic projection to this $q$-expansion, and simply throw away those parts of the coefficients that depend on $\vec{v}$.   

To make this more precise, we first expand
\[
E'(\vec{\tau} , 0 , \varphi_\mu) =    \frac{1}{ \Lambda ( 0, \chi ) }    \sum_{  \alpha  \in F }  a_F ( \alpha , \vec{v} ,  \mu)   \cdot q^\alpha  ,
\]
where 
\[
q^\alpha = e^{2\pi i \mathrm{Trace}_{F/\Q} ( \alpha \vec{\tau} ) }
\]
and $\mathrm{Trace}_{F/\Q}  : F \otimes_\Q \C \to \C$ is the usual trace.

For any totally positive $\alpha\in F$,  the set of places of $F$ at which $\mathscr{V}$ does not represent $\alpha$ is finite of even cardinality.  
As $\mathscr{V}$ is negative definite at one archimedean place, and positive definite at the others, it follows that the set of  nonarchimedean places
\begin{equation}\label{diff set}
\mathrm{Diff}(\alpha) = \{ \mbox{primes }\mathfrak{p} \subset \co_F : \mathscr{V}_\mathfrak{p} \mbox{ does not represent } \alpha\}
\end{equation}
is finite with odd cardinality.  In particular it is nonempty.
 Moreover, it is not hard to see that every $\mathfrak{p} \in \mathrm{Diff}(\alpha) $ is either inert or ramified in $E$.

\begin{proposition}\label{prop:eisenstein coefficients}
Let $F_+\subset F$ be the subset of totally positive elements.
If $\alpha \in F_+$ then 
\[
a_F (\alpha ,  \mu) \define a_F ( \alpha , \vec{v} , \varphi_\mu)
\]
is independent of $\vec{v}$.  Moreover,
\begin{enumerate}
\item
if $|\mathrm{Diff}(\alpha)|>1$ then $a_F(\alpha,\mu)=0$,
\item
if $\mathrm{Diff}(\alpha)=\{ \mathfrak{p}\}$ then 
\[
\frac{ a_F(\alpha,\mu)} { \Lambda(\chi,0) } \in \Q\cdot \log(\mathrm{N}(\mathfrak{p}) ) .
\]
\end{enumerate}
As for the $\alpha=0$ term,  there is a constant $a_F( 0, \mu  )\in \C$ such that 
\[
\frac{ a_F ( 0  , \vec{v} , \mu) }{  \Lambda(0,\chi) } 
 =  \frac{ a_F(0 , \mu) }{ \Lambda(0,\chi ) }  +  \varphi_\mu(0)  \log  \big( \mathrm{Norm}_{F/\Q}( \vec{v}) \big).
\]
This constant satisfies 
\[
\frac{a_F(0,\mu) }{ \Lambda(0,\chi) }  =
\begin{cases}
- 2  \frac{  L'(0,\chi)  }{ L(0,\chi) }  - \log \left| \frac{  D_{E}  }{  D_{F} }\right|  +d  \cdot  \log(4\pi e^\gamma)  & \mbox{if }\mu=0 \\
0 & \mbox{if }\mu\neq 0
\end{cases}
\]
 up to a $\Q$-linear combination of $\{ \log(p) :  p \in \Sigma_\mathrm{bad} \}$, where 
  $\Sigma_\mathrm{bad}$ is the set of bad primes of Definition \ref{def:bad primes}.
\end{proposition}

\begin{proof}
The  claims concerning $a_F(\alpha,\mu)$ with $\alpha \in F_+$  follow from \cite[Proposition 4.6]{BKY}.
The claims about the constant term follow from  \cite[Proposition 6.2.3]{AGHMP-2} and \cite[Proposition 7.8.2]{AGHMP-2}.
\end{proof}

Now form the formal $q$-expansion
\[
\mathcal{E}_F(\vec{\tau} ,  \mu) =
a_F(0,  \mu) + \sum_{ \alpha\in F_+ } a_F(\alpha,  \mu) \cdot q^\alpha.
\]
 This formal $q$-expansion is the lazy man's holomorphic projection indicated above. We stress that this is not a modular form, and we have no interest in issues of convergence.  
If we  formally restrict $\mathcal{E}_F(\vec{\tau} ,  \mu)$ to the diagonally embedded upper half-plane  $\mathcal{H} \subset \mathcal{H}_F$, the result is a formal $q$-expansion
\[
\mathcal{E} ( \tau , \mu) = \sum_{ m \ge 0} a(m,\mu)  \cdot q^m
\]
where $a(0,\mu) = a_F(0,\mu)$,  and
\begin{equation}\label{hilbert trace}
a( m ,\mu) =  \sum_{ \substack{   \alpha\in F_+ \\ \mathrm{Trace}_{F/\Q}(\alpha) = m      } }    a_F(\alpha , \mu)
\end{equation}
for all rational numbers $m>0$.


\subsection{The arithmetic intersection formula}


As in \S \ref{ss:theta} (with $n=2d-2$), suppose that 
\begin{equation}\label{final input}
f (\tau) = \sum_{ m \gg -\infty} c (m) q^{-m}  \in M^!_{2-  d  }(\weil_L)
\end{equation}
has all $c(m,\mu)\in \Z$, and let 
\[
\widehat{\mathcal{Z}} (f) = ( \mathcal{Z}(f) , \Phi(f) )  \in \widehat{\mathrm{CH}}^1(\mathcal{M})
\]
be the arithmetic divisor of Theorem \ref{thm:green}.  
We wish to  computes its image under the arithmetic intersection 
\[
 [- : \mathcal{Y}_\Z ] : \widehat{\mathrm{CH}}^1(\mathcal{M}) \to \R.
\]

\begin{lemma}
For any rational number $m>0$ and any $\mu \in L^\vee / L$,  the fiber product
\[
\mathcal{Z}(m,\mu) \cap \mathcal{Y}_\Z \define \mathcal{Z}(m,\mu) \times_{\mathcal{M}} \mathcal{Y}_\Z
\]
has dimension $0$, and is supported in finitely many nonzero characteristics.  
In other words, $\mathcal{Y}_\Z$ intersects all special divisors properly.
\end{lemma}

\begin{proof}
This follows from  \cite[Proposition 7.6.1]{AGHMP-2}, which also gives information about which characteristics can appear in the fiber product.  
\end{proof}

The significance of the lemma is that, as per the discussion of \S \ref{ss:arakelov}, there is a decomposition
\begin{equation}\label{intersection decomp}
 [ \widehat{\mathcal{Z}} (f)  :  \mathcal{Y}_\Z  ] 
 =
  [ \mathcal{Z} (f)  :  \mathcal{Y}_\Z  ]_\mathrm{fin}  +  [ \Phi (f)  :  \mathcal{Y}_\Z  ]_\infty
\end{equation}
as the sum of a  finite part 
\[
 [ \mathcal{Z} (f)  :  \mathcal{Y}_\Z  ]_\mathrm{fin}
 = \sum_{  \substack{   m>0 \\ \mu \in L^\vee/L   }  }     c(-m,\mu) \cdot     [ \mathcal{Z} (m,\mu)   :  \mathcal{Y}_\Z  ]_\mathrm{fin}
\]
and an archimedean part
\[
 [ \Phi (f)  :  \mathcal{Y}_\Z  ]_\infty  = \frac{1}{2} \sum_{ y \in \mathcal{Y}_\Z(\C) } \frac{\Phi(f, y) }{ |\Aut(y) | }.
\]

The following theorem is the most technically difficult part of \cite{AGHMP-2}.

\begin{theorem}\label{thm:deformation}
For any rational number $m>0$ and any $\mu \in L^\vee/L$, the equality 
\[
\frac{ 2d }{  \deg_\C( \mathcal{Y}_\Z)   } \cdot [ \mathcal{Z} (m,\mu)  :  \mathcal{Y}_\Z  ]_\mathrm{fin}
= -  \frac{a(m,\mu)}{ \Lambda(0,\chi) }
\]
holds up to a $\Q$-linear combination of $\{ \log(p) :  p \in \Sigma_\mathrm{bad} \}$.
Here $\Sigma_\mathrm{bad}$ is the set of bad primes of Definition \ref{def:bad primes}. 
\end{theorem}

\begin{proof}
We  only give the barest hint of the idea.

To any scheme $S$ equipped with a morphism $S\to \mathcal{M}$, we associated in \S \ref{ss:integral divisors} a positive definite quadratic space $V(\mathcal{A}_S)_\Q$ over $\Q$, along with a subset
\[
V_\mu(\mathcal{A}_S) \subset V(\mathcal{A}_S)_\Q.
\]
Let $Q$ denote the quadratic form on $V(\mathcal{A}_S)_\Q$.
If the morphism $S\to \mathcal{M}$ factors through $\mathcal{Y}_\Z \to \mathcal{M}$, then $V(\mathcal{A}_S)_\Q$ carries more structure: it is an $F$-vector space (in fact, an $E$-vector space), and carries an $F$-valued quadratic form $\mathscr{Q}$ satisfying 
\[
Q= \mathrm{Trace}_{F/\Q} \circ \mathscr{Q} .
\]
See \cite[Corollary 5.4.6]{AGHMP-2}.

Suppose $\alpha\in F_+$.
Recalling  that $\mathcal{Y}_\Z$ is just the stack $\mathcal{Y}$ of \S \ref{ss:big cm}, but viewed as a stack over $\Z$ rather than $\iota_0(\co_E) \iso \co_E$, we define a stack over $\mathcal{Y}$ whose functor of points assigns to any morphism $S\to \mathcal{Y}$ the set 
\[
\mathcal{Z}_F(\alpha , \mu)(S) = \{ x\in V_\mu(\mathcal{A}_S) :   \mathscr{Q} (x) = \alpha \}.
\]
By comparing with the definition (\ref{special cycle}) of $\mathcal{Z}(m,\mu)$, we obtain a canonical decomposition
\[
\mathcal{Z}(m,\mu) \times_\mathcal{M} \mathcal{Y} \iso   \bigsqcup_{ \substack{   \alpha\in F_+ \\ \mathrm{Trace}_{F/\Q}(\alpha) = m      } }    \mathcal{Z}_F(\alpha , \mu)
\]
of stacks over $\co_E$,  intended to mirror to the relation (\ref{hilbert trace}).  
Using (\ref{finite decomp by points}), the finite part of the arithmetic intersection becomes
\begin{align*}
[ \mathcal{Z} (m,\mu)  :  \mathcal{Y}_\Z  ]_\mathrm{fin}
& =
\sum_p  \log(p) \left(
\sum_{   y\in (\mathcal{Z}(m,\mu)\cap \mathcal{Y}_\Z)( \F_p^\alg)     }
 \frac{ \mathrm{length}\big( \co^\mathrm{et}_{ ( \mathcal{Z}(m,\mu)\cap \mathcal{Y}_\Z)  ,y  }  \big)  }{ | \Aut(y) | } \right) \\
 & =
 \sum_{\alpha \in F_+}
\sum_{   \substack{ \mathfrak{p} \subset \co_F \\  y\in \mathcal{Z}_F(\alpha,\mu) ( \F_\mathfrak{p}^\alg)     }}
 \frac{ \mathrm{length}\big( \co^\mathrm{et}_{ \mathcal{Z}_F(\alpha,\mu)  ,y  }  \big)  }{ | \Aut(y) | }   \cdot \log( \mathrm{N}(\mathfrak{p}) )   ,
 \end{align*}
 where $\F_\mathfrak{p}^\alg$ is an algebraic closure of $\co_F/\mathfrak{p}$.

Recall the set $\mathrm{Diff}(\alpha)$ primes of $\co_F$ defined by (\ref{diff set}).
If $| \mathrm{Diff}(\alpha)| >1$ then $\mathcal{Z}_F(\alpha,\mu)=\emptyset$.
On the other hand, if $\mathrm{Diff}(\alpha)=\{ \mathfrak{p}\}$  contains a single prime $\mathfrak{p}\subset \co_F$,  then  $\mathcal{Z}_F(\alpha,\mu)$ has dimension $0$, and is supported at the unique prime of $\co_E$ above $\mathfrak{p}$.   See \cite[Proposition 7.6.1]{AGHMP-2}.

Now suppose that $\mathrm{Diff}(\alpha) = \{\mathfrak{p}\}$ where  $\mathfrak{p}$ lies above a rational prime $p\not\in \Sigma_\mathrm{bad}$.   
Every \'etale local ring of $\mathcal{Z}_F(\alpha,\mu)$ has the same length 
\[
\ell_\mathfrak{p}(\alpha) = \mathrm{length}\big(  \co^\mathrm{et}_{ \mathcal{Z}_F(\alpha,\mu)  ,y  }  \big), 
\]
which  is given by an explicit formula in terms of $\ord_\mathfrak{p}(\alpha)$, and satisfies 
\[
 \ell_\mathfrak{p}(\alpha)    \log( \mathrm{N}(\mathfrak{p})  )
  \sum_{  y\in \mathcal{Z}_F(\alpha,\mu) ( \F_\mathfrak{p}^\alg)     }  
 \frac{ 1 }{ | \Aut(y) | }     
 =  -  \frac{  \deg_\C( \mathcal{Y}_\Z)   } { 2d }   \cdot  \frac{a_F(\alpha,\mu)}{ \Lambda(0,\chi) }.
\]
This calculation is carried out in \cite[\S 7]{AGHMP-2}, and lies at the core of the proof of the average Colmez conjecture.  See \cite[Theorem 7.7.4]{AGHMP-2} and  \cite[Theorem 7.8.1]{AGHMP-2}.
 The lengths of the local rings are determined using integral $p$-adic Hodge theory to compute formal deformation spaces of $p$-divisible groups with extra structure.   The summation on the left, which counts geometric points on the stack $\mathcal{Z}_F(\alpha,\mu)$,  can be rewritten so that it counts vectors of  norm $\alpha$ in lattices contained in a  \emph{nearby} positive definite $F$-quadratic space ${}^\mathfrak{p}\mathscr{V}$ obtained by modifying the quadratic space $\mathscr{V}$ at $\mathfrak{p}$ and the unique archimedean place at which it is negative definite.  This essentially realizes the summation  as a  theta series coefficient, and the connection with the Eisenstein series coefficient $a_F(\alpha,\mu)$ comes via the Siegel-Weil formula.  This follows the method for counting points used in \cite{HY} and \cite{Ho12} for Hilbert modular surfaces and unitary Shimura varieties, respectively.

Combining the above with Proposition \ref{prop:eisenstein coefficients}   shows that
\begin{align*}
  [ \mathcal{Z} (m,\mu)  :  \mathcal{Y}_\Z  ]_\mathrm{fin}
 & =
\sum_{   \substack{\alpha \in F_+\\  \mathfrak{p} \subset \co_F \\  y\in \mathcal{Z}_F(\alpha,\mu) ( \F_\mathfrak{p}^\alg)     }}
 \frac{   \ell_\mathfrak{p}(\alpha)    \cdot \log( \mathrm{N}(\mathfrak{p}) )   }{ | \Aut(y) | }    =  -  \frac{  \deg_\C( \mathcal{Y}_\Z)   } { 2d }  \cdot \frac{a_F(\alpha,\mu)}{ \Lambda(0,\chi) } 
 \end{align*}
up to a  linear combination of $\{ \log(p) : p \in \Sigma_\mathrm{bad} \}$, as desired.
%
\end{proof}

The archimedean part of (\ref{intersection decomp}) was computed by Bruinier-Kudla-Yang, following methods of  \cite{BY09} and \cite{Sch}.
Similar  formulas on Hilbert modular surfaces and modular curves appeared earlier in work of Bruinier-Yang \cite{BY06} and Gross-Zagier \cite{GZ}, respectively, but  those calculations used different  methods.

\begin{theorem}[Bruinier-Kudla-Yang \cite{BKY}]\label{thm:BKY}
For any  $f$ as in (\ref{final input}),  the regularized theta lift of \S \ref{ss:theta} satisfies
\[
\frac{d}{  \deg_\C( \mathcal{Y}_\Z )    } \sum_{ y\in \mathcal{Y}_\Z  (\C) }   \frac{ \Phi(f,y) }{ \Aut(y) }
=  \sum_{ \substack{  \mu \in L^\vee/L   \\ m\ge 0  } }  \frac{ a(m,\mu) \cdot c(-m,\mu) } {\Lambda(0,\chi)}.
\]
\end{theorem}

\begin{remark}
The actual result proved in  \cite{BKY}  is stronger than Theorem \ref{thm:BKY}. 
Whereas we only need regularized theta lifts of weakly holomorphic forms, the general result proved in [loc.~cit.]   includes 
regularized theta lifts of harmonic weak Maass forms.   This greater generality is not needed for the proof of the averaged Colmez conjecture.
\end{remark}

The weakly holomorphic form (\ref{final input}) may be chosen so that $c(0,0)\neq 0$.
Putting together the finite and archimedean intersection calculations above  yields the following arithmetic intersection formula.

\begin{corollary}\label{cor:intersection}
 For any $f$ as in (\ref{final input}), the equality 
\[
\frac{ 2d }{  c(0,0)   } \cdot \frac{  [ \widehat{\mathcal{Z}} (f)  :  \mathcal{Y}_\Z  ] }{  \deg_\C( \mathcal{Y}_\Z)  }
=
- 2  \frac{  L'(0,\chi)  }{ L(0,\chi) }  - \log \left| \frac{  D_{E}  }{  D_{F} }\right|  +d  \cdot  \log(4\pi e^\gamma) 
\]
holds up to a $\Q$-linear combination of $\{ \log(p) :  p \in \Sigma_\mathrm{bad} \}$.
\end{corollary}

\begin{proof}
Combining Theorems \ref{thm:deformation} and \ref{thm:BKY} with the decomposition (\ref{intersection decomp}) shows that
\[
2d \cdot \frac{ [ \widehat{\mathcal{Z}} (f)  :  \mathcal{Y}_\Z  ]  }{  \deg_\C( \mathcal{Y}_\Z)   } 
= \sum_{\mu \in L^\vee /L}  \frac{a(0,\mu) \cdot c(0,\mu) }{ \Lambda(0,\chi) }
\]
holds up to a $\Q$-linear combination of $\{ \log(p) :  p \in \Sigma_\mathrm{bad} \}$.
Combining this with Proposition \ref{prop:eisenstein coefficients} completes the proof.
\end{proof}


\subsection{Putting everything together}


We can now prove Theorem \ref{thm:average colmez} of the introduction.

\begin{theorem}
For any CM field $E$ with totally real subfield $F\subset E$ of degree $[F:\Q]=d$, we have the equality
\begin{equation}\label{final form}
\frac{1}{2^d} \sum_\Phi h^\Falt_{(E,\Phi)} =
-  \frac{1}{2} \cdot \frac{ L'(0,\chi)  }{ L(0,\chi) }  -   \frac{1}{4}  \cdot  \log \left| \frac{  D_{E}  }{  D_{F} }\right| 
-  \frac{ d   \log(2\pi)}{ 2 }  ,
\end{equation}
where the summation is over all CM types $\Phi\subset \Hom(E,\C)$.
\end{theorem}

\begin{proof}
Recall  the finite set of bad primes $\Sigma_\mathrm{bad}$ of Definition \ref{def:bad primes}.
Theorem \ref{thm:taut to faltings} and  Corollary \ref{cor:taut to special} give us the equalities
\[
\frac{1}{ 2^d   }\sum_\Phi h^\Falt_{(E,\Phi)}  
 =  \frac{d}{2} \cdot \frac{  [  \widehat{\taut}   :    \mathcal{Y}_\Z   ]  }  {  \deg_\C( \mathcal{Y}_\Z )    }     -   \frac{d}{2} \cdot   \log(2\pi)   
\]
and
\[
\frac{d}{2} \cdot   \frac{  [  \widehat{\taut}   :    \mathcal{Y}_\Z  ]  }  {  \deg_\C( \mathcal{Y}_\Z)      }    
 =\frac{d}{2c(0,0)} \cdot  \frac{  [ \widehat{\mathcal{Z}}(f) : \mathcal{Y}_\Z ] }  {   \deg_\C( \mathcal{Y}_\Z)  }  
- \frac{d}{4} \cdot  \log ( 4 \pi  e^\gamma )   ,
\]
up to a $\Q$-linear combination of $\{ \log(p) :  p \in \Sigma_\mathrm{bad} \}$.
Comparing these with Corollary \ref{cor:intersection} shows that the desired equality (\ref{final form})
holds up to a $\Q$-linear combination of $\{ \log(p) :  p \in \Sigma_\mathrm{bad} \}$. 

We now write
\[
\frac{1}{2^d} \sum_\Phi h^\Falt_{(E,\Phi)} =
-  \frac{1}{2} \cdot \frac{ L'(0,\chi)  }{ L(0,\chi) }  -   \frac{1}{4}  \cdot  \log \left| \frac{  D_{E}  }{  D_{F} }\right| 
-  \frac{ d   \log(2\pi)}{ 2 }
+ \sum_p b_E(p) \log(p) 
\]
for   constants $b_E(p) \in  \Q$  satisfying $b_E(p)=0$ whenever $p \not\in \Sigma_\mathrm{bad}$.
By the $\Q$-linear independence of $\{ \log(p) \}$, the constants  $b_E(p)$ are uniquely determined by this relation, and each depends only on $E$ and $p$.

In particular $b_E(p)$ does not depend on the choice of   $\xi \in F^\times$ used to define the quadratic space $(V,Q)$, or on the choice of maximal lattice $L\subset V$. 
 The set $\Sigma_\mathrm{bad}$, however, depends very much on these choices.  
 According to Proposition 9.5.2 of \cite{AGHMP-2}, for any given prime $p$ one can choose the auxiliary data in such a way that $p\not\in \Sigma_\mathrm{bad}$, and hence $b_E(p)=0$.
\end{proof}

The proof above is somewhat simpler than what is presented in \S 9.5 of \cite{AGHMP-2}, where some calculations are carried out after enlarging the Shimura variety $\mathcal{M}$ to some $\mathcal{M}^\beef$ as in \S \ref{ss:general integral}.  The point is that at the time \cite{AGHMP-2} was written, the results of \cite{HMP} were not yet available, and so the equality of divisors of Theorem \ref{thm:borcherds} was only known 
(by work of H\"ormann, as in  Remark \ref{rem:hormann}) 
up to a sum of  divisors supported in  characteristics dividing $[L^\vee : L]$.  
The arguments of \cite{AGHMP-2} were written in such a way that they only require this weaker result.

\bibliographystyle{amsalpha}

\newcommand{\etalchar}[1]{$^{#1}$}
\providecommand{\bysame}{\leavevmode\hbox to3em{\hrulefill}\thinspace}
\providecommand{\MR}{\relax\ifhmode\unskip\space\fi MR }
\providecommand{\MRhref}[2]{%
  \href{http://www.ams.org/mathscinet-getitem?mr=#1}{#2}
}
\providecommand{\href}[2]{#2}

\end{document}